%% file: main_v4.tex
\RequirePackage{fix-cm}
\documentclass[smallextended]{svjour3}       
\smartqed  
\input{macros}

\usepackage{graphicx}
\usepackage[colorlinks=true,breaklinks=true,bookmarks=true,urlcolor=blue,
     citecolor=blue,linkcolor=blue,bookmarksopen=false,draft=false]{hyperref}
\def\email#1{\href{mailto:#1}{#1}}
\usepackage{xcolor}      

\usepackage{epstopdf}
\newcommand{\h}[1]{\mathbf{#1}}
\newcommand{\R}{\mathbb{R}}
\providecommand{\norm}[1]{\left\lVert#1\right\rVert}
\DeclareMathOperator*{\supp}{supp}

\newcommand{\jq}[1]{{\color{magenta}#1 --JQ}}

\begin{document}

\title{A Novel Regularization Based on the Error Function for Sparse Recovery} 
\titlerunning{Error Function for Sparse Signal Recovery}        
\author{Weihong Guo \and Yifei Lou \and Jing Qin  \and Ming Yan}
\institute{
	W. Guo\at
	Department of Mathematics, Applied Mathematics and Statistics\\
	Case Western Reserve University, Cleveland, OH 44106.\\
	\email{wxg49@case.edu}.\and
Y. Lou \at
Department of Mathematical Sciences\\
The University of Texas at Dallas, Richardson, TX 75080. \\
\email{yifei.lou@utdallas.edu}.
\and
J. Qin\at
 Department of Mathematics\\
 University of Kentucky, Lexington, KY 40506.\\
 \email{jing.qin@uky.edu}.
\and
M. Yan \at
Department of Computational Mathematics, Science and Engineering\\
Department of Mathematics\\
Michigan State University, East Lansing, MI 48824.\\
\email{yanm@math.msu.edu}.
}

\date{Received: date / Accepted: date}

\maketitle

\begin{abstract}
Regularization plays an important role in solving ill-posed problems by adding extra information about the desired solution, such as sparsity. Many regularization terms usually involve some vector norm, e.g., $L_1$ and $L_2$ norms.
In this paper, we propose a novel regularization framework that uses the error function to approximate the unit step function. It  can be considered as a surrogate function for the $L_0$ norm. The asymptotic behavior of the error function with respect to its intrinsic parameter indicates that the proposed regularization can approximate the standard $L_0$, $L_1$ norms as the parameter approaches to $0$ and $\infty,$ respectively. Statistically, it is also less biased than the $L_1$ approach. We then incorporate the error function into either a constrained or an unconstrained model when recovering a sparse signal from an under-determined linear system. Computationally, both problems can be solved via an iterative reweighted $L_1$ (IRL1) algorithm with guaranteed convergence. A large number of experimental results demonstrate that the proposed approach outperforms the state-of-the-art methods in various sparse recovery scenarios.

\end{abstract}


\section{Introduction}\label{sec:intro}
In these days, ``big data'' is ubiquitous due to developments and advancements of science and technologies. But on the other hand, one often faces  ``small data,'' when the amount of data that can be transmitted is limited by technical or economic restrictions.  For example, when a patient undergoes the computed tomography (CT) scanning, the amount of measurements that can be recorded can not exceed the maximum safe radiation dosage.
Mathematically speaking, a small data problem  corresponds to an under-determined (linear) system, where the number of measurements is considerably smaller than the ambient dimension. In this case, reasonable assumptions shall be taken into account and formulated as regularizations to refine the desired solution space.

In data science, a signal of interest is often assumed to be sparse (i.e. having a few non-zero elements) either by itself or after a linear transformation. One popular signal recovery technique based on sparsity is referred to as \textit{compressed sensing} (CS), coined by David Donoho \cite{donoho2006compressed}. CS enables data compression to facilitate data storage and transmission, as only a small portion of data, which are non-zero, is being processed. In order to find the sparsest signal, it is natural to minimize the $L_0$ norm\footnote{Note that $\|\cdot\|_0$ is a pseudo-norm, but is often called as the $L_0$ norm.}, i.e., the number of nonzero entries in a vector.
Unfortunately, the $L_0$ minimization is NP-hard \cite{natarajan95}, as it involves combinatorial search that is time-consuming, especially in high-dimensional spaces.  One of the most popular approaches in CS is to replace the $L_0$ norm by the convex $L_1$ norm, which often gives satisfactory results. This $L_1$ convex relaxation technique has been applied in many different fields such as geology and geophysics \cite{santosaS86}, Fourier transform spectroscopy \cite{mammone83}, and ultrasound imaging \cite{papoulisC79}. One major tool for analyzing CS algorithms is the restricted isometry property (RIP) \cite{CRT}, which provides a sufficient condition for exact recovery of a sparse signal by minimizing the $L_1$ norm.

The $L_1$ minimization in CS is closely related to least absolute shrinkage and selection operator (LASSO) \cite{tibshirani96lasso} in statistical learning. Assume that the data is generated by a linear regression model polluted by Gaussian noise, $\h b=A\h x+\varepsilon$, where each row of $A$ is a sample of feature vectors,  and $\h b$, $\varepsilon$ are response and noise, respectively. In this setting, one aims to find a sparse vector $\h x$ consisting of model coefficients, which is a reasonable assumption, since only a few features contribute to the response. However, Fan and Li  \cite{fan2001variable} pointed out that LASSO (or $L_1$) is biased towards  large coefficients. To mitigate the estimation bias, they proposed a nonconvex regularization, called smoothly clipped absolute deviation (SCAD). Later, many other nonconvex regularizations have emerged in statistics, such as capped $L_1$  (CL1) \cite{zhang2009multi,shen2012likelihood}, transformed $L_1$ (TL1) \cite{lv2009unified}, and minimax concave penalty (MCP) \cite{zhang2010nearly}. Some of these models have been adopted for sparse signal recovery \cite{louYX16,zhangX17,zhangX18}.

Nonconvex regularization terms can be further categorized into two groups: smooth and nonsmooth. In particular, Capped $L_1$, SCAD, and MCP are nonsmooth. Specifically their proximal functions are not continuous, which leads to numerical instability from the algorithmic point of view. Transformed $L_1$ is smooth except at zero and its proximal function is continuous. However, it has the bias issue as well, while  CL1, SCAD, and MCP yield an unbiased estimate, as they are constant for large component. We aim to propose a nonconvex regularization, which is smooth and less biased compared to $L_1$ and TL1; see Figure~\ref{fig:obj_fun} for a comparison among these regularization terms.

In this paper, we propose a novel nonconvex regularization based on the ERror Function (ERF) to promote sparsity. It is motivated from a graph-based approach \cite{bai2018graph} to enforce a bi-modal weight distribution when reconstructing a skeleton image.
We discover that their numerical scheme is equivalent to minimizing the error function via the iterative reweighted $L_1$ (IRL1) algorithm \cite{candes2008enhancing}. As a good approximation to the
Heaviside step function (or the unit step function), the error function can serve as a surrogate function for the $L_0$ norm, which has not been considered in the CS literature to the best of our knowledge. The major contributions of this paper are three-fold:
\begin{enumerate}
	\item[(a)] We propose a novel regularization based on the error function for sparse signal recovery and establish its connections to the standard $L_0$, $L_1$ regularizations;
	\item[(b)] We adapt the IRL1 algorithms to solve the proposed model in either a constrained or an unconstrained formulation with guaranteed convergence;
	\item[(c)] We conduct extensive experiments to demonstrate the superior performance of the proposed approaches.
\end{enumerate}

The rest of the paper is organized as follows. We review some existing models and related algorithms for sparse recover in Section~\ref{sect:review}. The proposed regularization is described in Section~\ref{sect:reg_ERF}, followed by numerical schemes in Section~\ref{sect:alg}. We present experimental results in Section~\ref{sect:exp}, showing that the proposed approaches outperform the state of the art in sparse recovery. Finally, conclusions and future works are given in Section~\ref{sect:conclude}.

\section{Preliminaries}\label{sect:review}
Throughout the paper, we use bold uppercase letters to denote matrices, bold lowercase letters to denote vectors, and lower case letters to denote vector or matrix entries, e.g., a vector $\vx$ with its $j$-th component by $x_j$. The set of all $n$-dimensional real vectors is denoted by $\mathbb{R}^n$. The $L_p$ norm of a vector $\vx\in\mathbb{R}^n$ is defined as $\|\vx\|_p=(\sum_{j=1}^n|x_j|^p)^{1/p}$ for $0<p\leq \infty$. The sign function applied to $\vx\in\mathbb{R}^n$ returns a vector, denoted by $\sign(\vx)$, whose $j$-th component is $x_j/|x_j|$ if $x_j\neq0$ and zero otherwise. Inequalities involving vectors are defined component-wise, e.g., $\vx\leq \vy$ meaning that each component of $\vx$ is less than or equal to the corresponding component of $\vy$. We use $\odot$ to denote the component-wise multiplication of two vectors. The set of all $m\times n$ real matrices is denoted by $\mathbb{R}^{m\times n}$. The kernel of a matrix $\vA\in\mathbb{R}^{m\times n}$ is defined as $\mathrm{ker}(\vA):=\{\vx\in\mathbb{R}^n:\vA\vx=\mathbf{0}\}$. The trace of a square matrix $\vA\in\R^{n\times n}$, denoted by $\tr(\vA)$, calculates the sum of all diagonal entries. 
$\vA_S$ is the submatrix with columns selected from the index set $S$ and $\vx_S$ is the subvector with components selected from the index set $S$.

\subsection{Sparsity-promoting models}
There are a variety of regularizers that can approximate the $L_0$ norm,  including $L_1$, $L_p$ with $0<p<1$ \cite{chartrand07,Xu2012}, capped $L_1$  \cite{zhang2009multi,shen2012likelihood,louYX16}, transformed $L_1$ \cite{lv2009unified,zhangX17,zhangX18},
 $L_1$-$L_2$ \cite{yinEX14,louYHX14}, and $L_1/L_2$ \cite{l1dl2,l1dl2accelerated}.
   In this paper, we focus on developing a separable regularization, which allows component-wise implementation to enhance computational efficiency. Besides the $L_1$ and $L_p$ with $0<p<1$, there are  other popular separable regularizations, some of which are listed as follows: for any $\vx\in\R^n$ and $a,\lambda,\gamma>0$
\begin{itemize}
		
	
	\item Capped $L_1$  (CL1):
	\[
	J^{\mathrm{CL1}}_a(\h x) := \sum_{j=1}^n\min\{|x_j|,a\} ;
	\]
	
	\item Transformed $L_1$ (TL1):
	\[
	J^{\mathrm{TL1}}_a(\h x) := \sum_{j=1}^n \dfrac {(a+1)|x_j|}{a+|x_j|};
	\]

	\item Smoothly clipped absolute deviation (SCAD): $$J^{\mathrm{SCAD}}_{\lambda,\gamma}(\h x) := \sum_{j=1}^n \Phi_{\lambda,\gamma}^{\mathrm{SCAD}}(x_j)$$ with
	\[
	\Phi_{\lambda,\gamma}^{\mathrm{SCAD}}(x_j)=\left\{
	\begin{aligned}
	&\lambda|x_j|,&&|x_j|\leq\lambda;\\
	&\frac{2\gamma\lambda|x_j|-|x_j|^2-\lambda^2}{2(\gamma-1)},&&
	\lambda <|x_j|\leq \gamma\lambda;\\
	&\frac{(\gamma+1)\lambda^2}2,&&|x_j|>\gamma\lambda;
	\end{aligned}
	\right.
	\]

	\item Minimax concave penalty (MCP):
	$$J^{\mathrm{MCP}}_{\lambda,\gamma}(\h x) := \sum_{j=1}^n \Phi_{\lambda,\gamma}^{\mathrm{MCP}}(x_j)$$ with
	\[
	\Phi_{\lambda,\gamma}^{\mathrm{MCP}}(x_j)=
	\left\{
	\begin{aligned}
	&\lambda |x_j|-\frac{|x_j|^2}{2\gamma},&& |x_j|\leq\gamma
	\lambda;\\
	&\frac12\gamma\lambda^2,&& |x_j|>\gamma\lambda;
	\end{aligned}
	\right.
	\]	
\end{itemize}

It is straightforward that $L_p$ converges to $L_0$ and $L_1$ as $p$ goes to $0$ and $1$, respectively. By letting $a=0$ in TL1 and using the standard assumption of $\frac 0 0 = 0$, $J_0^{\mathrm{TL1}}$ is equivalent to the $L_0$ norm. On the other hand, we have
\[
\lim_{a\rightarrow \infty} J_a^{\mathrm{TL1}}(\h x)=\sum_{j=1}^n \lim_{a\rightarrow \infty}
\dfrac {(1+\frac 1 a)|x_j|}{1+\frac 1 a |x_j|} = \sum_j |x_j| = \|\h x\|_1.\]
Therefore, $J_a^{\mathrm{TL1}}$ approaches to the $L_0$ and $L_1$ norms by letting $a\rightarrow 0$ and $\infty$, respectively.

Both SCAD and MCP are proposed to correct the estimation bias caused by the $L_1$ approach. One criterion for an unbiased function is that its derivative has a horizontal asymptote at zero, as suggested in SCAD \cite{fan2001variable}. As $\Phi_{\lambda,\gamma}^{\mathrm{SCAD}}$ and $\Phi_{\lambda,\gamma}^{\mathrm{MCP}}$ become constant for a relatively large variable, both SCAD and MCP estimates are unbiased. However, one major drawback of SCAD and MCP is that there are two model parameters involved, which causes difficulties in parameter selection. The parameter-free models include $L_1,$ $L_1$-$L_2$, and $L_1/L_2$, the last of which also has a scale-invariant property to mimic the $L_0$ norm.

\subsection{Optimization techniques}

A fundamental problem in CS is to find a sparse vector subject to an under-determined linear system,
\begin{equation}\label{eq:CS}
\min_{\h x\in \R^n} \|\h x\|_0 \quad \st \quad  \vA\h x=\h b,
\end{equation}
where $\vA\in\mathbb R^{m\times n} (m\ll n)$ is called a \emph{sensing matrix} and $\h b\in\mathbb R^m$ denotes a measurement vector.
 Cand{\`e}s \emph{et al.}~proposed an iterative algorithm for reweighted $L_1$ minimization (IRL1) \cite{candes2008enhancing} as follows,
\begin{equation}\label{eqn:con_rwl14l0}
\left\{\begin{array}{l}
 w_j^{k} = {1\over{|x_j^k|+\epsilon}}\\
\h x ^{k+1} = \arg\min_{\h x\in\R^n} \sum_{j=1}^n w^{k}_j|x_j| \quad \st \quad \vA\h x = \h b,\\
\end{array}
\right.
\end{equation}
where a positive parameter $\epsilon$ is introduced for the sake of stability. From the perspective of a majorization-minimization (MM) framework \cite{LHY}, the iteration \eqref{eqn:con_rwl14l0} is in fact to minimize the following problem
\begin{equation}\label{eqn:log}
\min_{\h x\in \R^n} \sum_{j=1}^n \log(|x_j|+\epsilon) \quad \st \quad  \vA\h x=\h b.
\end{equation}
The objective function in \eqref{eqn:log} is often called a \emph{log-sum penalty function}, denoted by
\[
J^{\mathrm{log}}_\epsilon(\h x):=\sum_{j=1}^n \Phi^{\mathrm{log}}_\epsilon(x_j) \quad
\mbox{where}\quad \Phi^{\mathrm{log}}_\epsilon(x) = \log(|x|+\epsilon).
\]
Since $J^{\mathrm{log}}_\epsilon$ is a concave function on $\mathbb{R}^n$, we have
\[
J^{\mathrm{log}}_\epsilon(\h x )\leq J^{\mathrm{log}}_\epsilon(\h x^k) + \langle \nabla J^{\mathrm{log}}_\epsilon(\h x^k), \h x-\h x^k\rangle.
\]
Instead of directly minimizing $J^{\mathrm{log}}_\epsilon$, the MM framework considers the following iteration scheme
\[
\h x^{k+1} = \arg\min_{\h x\in \R^N} J^{\mathrm{log}}_\epsilon(\h x^k) + \langle \nabla J^{\mathrm{log}}_\epsilon(\h x^k), \h x-\h x^k\rangle,
\]
which is equivalent to \eqref{eqn:con_rwl14l0}.

The iterative reweighed algorithms are generalized in \cite{ochs2015iteratively}, where the authors considered a certain class of nonsmooth and nonconvex functions of the form
\begin{eqnarray}
\min_{\h x\in X} F_1(\h x)+ F_2(G(\h x)),
\end{eqnarray}
with a convex function $F_1$, a coordinate-wise convex function $G$,  a concave function $F_2$, and a feasible set  $X\subseteq\R^n$. The IRL1 algorithm can be expressed as
\begin{equation}\label{eqn:con_rwl14general}
\left\{\begin{array}{l}
w_j^{k} \in \partial F_2(\h y) \quad \mbox{with}\ \ \h y= G(\h x)\\
\h x ^{k+1} = \arg\min_{\h x\in X} F_1(\h x)+ \langle \h w^{k}, G(\h x)\rangle,\\
\end{array}
\right.
\end{equation}
where $\partial F_2$ denotes the subgradient of $F_2$.

\section{Proposed regularization} \label{sect:reg_ERF}

We propose a novel regularization to promote sparsity,
\begin{equation}\label{eqn:erf}
J^{\mathrm{ERF}}_{\sigma}(\h x) :=  \sum_{j=1}^n\Phi^{\mathrm{ERF}}_\sigma(|x_j|) \quad \mbox{with} \quad \Phi^{\mathrm{ERF}}_\sigma(x)=\int_0^{x}e^{-\tau^2/\sigma^2}d\tau,
\end{equation}
where $\h x\in\R^n$, $\sigma>0$. Note that the standard error function is defined as
\[
{\mathrm{erf}(x)} = \frac 2 {\sqrt\pi} \int_0^{x}e^{-\tau^2}d\tau,
\]
and hence $\Phi_\sigma^{\mathrm{ERF}}$ is a scaled error function in that
\begin{equation}
\Phi^{\mathrm{ERF}}_\sigma(x) = \sigma\int_0^{\frac{x}{\sigma}}e^{-\tau^2}d\tau = \frac{\sigma\sqrt{\pi}}{2} \mathrm{erf}\left(\frac{x}{\sigma}\right).
\end{equation}
We refer our model \eqref{eqn:erf}  as the ``ERF'' regularization. In what follows, we omit the superscript ``ERF'' in $J_\sigma$ and $\Phi_\sigma$, when the context clearly refers to the proposed regularization.

\subsection{Properties}

We list some useful properties about $\Phi_\sigma$ and $J_\sigma$, especially the asymptotical behaviors of $J_\sigma$ as characterized  in Theorem~\ref{thm:asym}.
\begin{itemize}
	\item The derivative of $\Phi_\sigma$ at $x$ is given by
	\begin{equation}\label{prop_erf}
	\frac{d}{dx}\Phi_\sigma(x) = \exp(-\frac{x^2}{\sigma^2}).
	\end{equation}

	\item The upper/lower bounds of $\Phi_\sigma$ are given by
	\begin{eqnarray}
	c\sqrt{1-e^{-ax^2}}\leq \Phi_\sigma(x)\leq c \sqrt{1-e^{-bx^2}},\quad \forall\,x\in\R,
	\end{eqnarray}
	where $a=1/\sigma^2$, $b=\pi/4\sigma^2$, and $c= \frac{\sigma\sqrt{\pi}}{2}$ based on the lower and upper bounds of the standard error function \cite{chu1955bounds}.

	\item  $J_\sigma$ is concave on $\R^n$, i.e., for any $t\in[0,1]$ and any $\h x,\h y\in\R^n$
	\[
	J_\sigma(t\h x+(1-t)\h y) \geq t J_\sigma(\h x)+(1-t)J_\sigma(\h y).
	\]

	\item $J_\sigma$ is subadditive or satisfies the triangle inequality, i.e.,
	\[
	J_\sigma(\h x+\h y) \leq J_\sigma(\h x)+J_\sigma(\h y),\quad \forall\,\h x,\h y\in\R^n.
	\]
    In addition, if $\vx,\vy\in\mathbb{R}^n$ have disjoint supports, then
\[
J_\sigma(\vx+\vy)=J_\sigma(\vx)+J_\sigma(\vy),
\]
which serves as a key in studying the $J_\sigma$-regularized minimization problem.
\end{itemize}




\begin{theorem}\label{thm:asym}
	For any nonzero vector $\h x\in\R^n$, we have
	\begin{enumerate}
		\item[(a)] $J_\sigma(\vx)\rightarrow \|\vx\|_1,$ as $\sigma\rightarrow +\infty$;
		\item[(b)] $J_\sigma(\vx)/\sigma \rightarrow {\sqrt{\pi}\over 2} \|\vx\|_0,$ as $\sigma\rightarrow 0^+$.
	\end{enumerate}
\end{theorem}

\begin{proof}
Since $J_\sigma(\cdot)$ is separable with respect to each component of $\h x$, it suffices to discuss the limits for a scalar. For a real number $x\neq 0$, we let $t=x/\sigma$ which approaches to zero as $\sigma \rightarrow +\infty$ and hence
we have
\begin{equation}\label{prop_sigma_inf}
\lim_{\sigma\rightarrow+\infty}  {\Phi_\sigma(x)\over x} = \lim_{t\rightarrow 0}   {\int_0^te^{-x^2}dx\over t} = 1.
\end{equation}
The last equality is based on the l'Hospital's rule.
When $x=0$, it is obvious that $\Phi(x/\sigma)=0=x$ and thereby $J_\sigma(\vx)\rightarrow \|\vx\|_1$ as $\sigma\rightarrow +\infty$.

On the other hand, we have $\Phi_\sigma(0)=0$ and
\begin{equation}
\lim_{\sigma\rightarrow 0^+}{\Phi_\sigma(x)\over\sigma} = {\sqrt{\pi}\over 2}, \quad \forall\, x\neq 0.
\end{equation}
Therefore, $J_\sigma(\h x)/\sigma\rightarrow \frac {\sqrt{\pi}} 2\|\h x\|_0$ as $\sigma\rightarrow 0^+.$ \qed
\end{proof}

\begin{figure}
	\centering
	\includegraphics[width=0.7\textwidth]{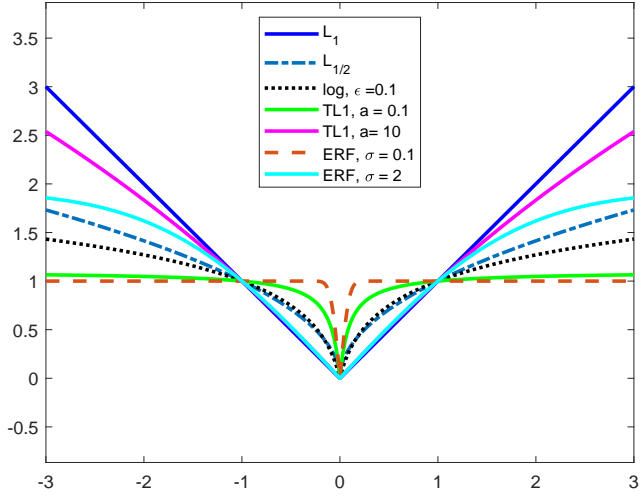}
	\caption{The objective functions of  various sparsity promoting models, all of which are scaled to attain the point $(1,1)$. The proposed ERF gives the best approximation to the $L_0$ norm for a  small value of $\sigma$ and is also relatively ``unbiased,'' compared to other models.} \label{fig:obj_fun}
\end{figure}

Figure~\ref{fig:obj_fun} shows the objective functions of various sparsity promoting models. We scale them to attain the point $(1,1)$ in order to have a better visual comparison. It indicates that the proposed ERF gives the best approximation to the $L_0$ norm for a  small value of $\sigma$ and is also relatively ``unbiased,'' compared to other models.

\subsection{Proximal operator}

Given a function $f: \mathbb R^n\rightarrow \mathbb R\cup \{+\infty\}$, the proximal operator \cite{parikh2014proximal} $\mathbf{prox}^f_{\mu}:\mathbb R^n\rightarrow \mathbb R^n$ of $f$ with a parameter $\mu>0$ is defined by
\begin{equation}\label{eq:prox-def}
\mathbf{prox}^f_\mu (\h v) = \arg\min_{\h x} \Big(\mu f(\h x)+\frac 1 2 \|\h x-\h v\|_2^2\Big).
\end{equation}
If $f$ is the $L_1$ norm, then the corresponding proximal operator is the \textit{soft shrinkage} operator, defined by
\begin{equation}\label{eq:soft_shrinkage}
\mbox{shrink}_\mu(\h v) = \left\{
\begin{array}{lll}
\h v-\mu ,&\ & \h v>\mu,\\
\mathbf{0}, &\ & |\h v|\leq \mu,\\
\h v+\mu ,&\ & \h v<-\mu.
\end{array}
\right.
\end{equation}
Due to its component-wise calculation, this operator is a key to make many $L_1$ minimization algorithms efficient. The proximal operator for the $L_0$ norm with parameter $\mu$ is given by the \textit{hard thresholding}
\begin{equation}\label{eq:hard_threshold}
\mbox{thresh}_\mu(\h v) = \left\{
\begin{array}{lll}
\h v, &\ & |\h v|>\mu,\\
0, &\ & |\h v|\leq \mu.
\end{array}
\right.
\end{equation}
As for TL1, its proximal operator \cite{zhangX17} can be expressed as
\begin{equation}\label{eq:TL1prox}
\mathbf{prox}^{\mathrm{TL1}}_{\mu}(\h v) = \left\{
\begin{array}{lll}
\sign(\h v)\Big[\frac 2 3(a+|\h v|)\cos(\frac{\varphi(\h v)}3)-\frac{2a}3 + \frac{|\h v|}3\Big], &\ & \h v>\mu,\\
0, &\ & |\h v|\leq \mu,\\
\end{array}
\right.
\end{equation}
where $\varphi(\h v) =\arccos\big(1-\frac {27\mu a(a+1)}{2(a+|\h v|^3)}\big)$.

Next we derive the proximal operator for the proposed ERF model. The optimality condition of \eqref{eq:prox-def} reads as
\begin{equation*}\label{eq:prox-optimality_ERF}
\h v\in \mu \partial f(\h x) + \h x = \mu \exp\big(-\frac{\h x^2}{\sigma^2}\big)\partial |\h x| + \h x.
\end{equation*}
When $|v_i|\leq \mu$, we have $x_i=0$. Otherwise the optimality condition becomes 
\begin{equation*}
v_i= \mu \exp\big(-\frac{x_i^2}{\sigma^2}\big)\sign(v_i) + x_i.
\end{equation*}
We can  find the solution via the Newton's iteration.

We plot the proximal operators for $L_0,L_1$, TL1 with $a=0.1, 10$, and ERF with $\sigma=0.1,2$ in Figure~\ref{fig:proximal}. Both TL1 and ERF provide the asymptotic approximations to $L_0$ and $L_1$ when varying their intrinsic parameter. The estimation bias issue can be illustrated by whether the proximal operator approaches to the diagonal line $y=x$ when the magnitude of $x$ increases. In this sense, the plots indicate that ERF causes less bias than $L_1$ and TL1.

\begin{figure}
	\centering
\setlength{\tabcolsep}{1pt}
	\begin{tabular}{ccc}
		$L_0$ & TL1, $a=0.1$ & ERF, $\sigma = 0.1$\\
	\includegraphics[width=0.33\textwidth]{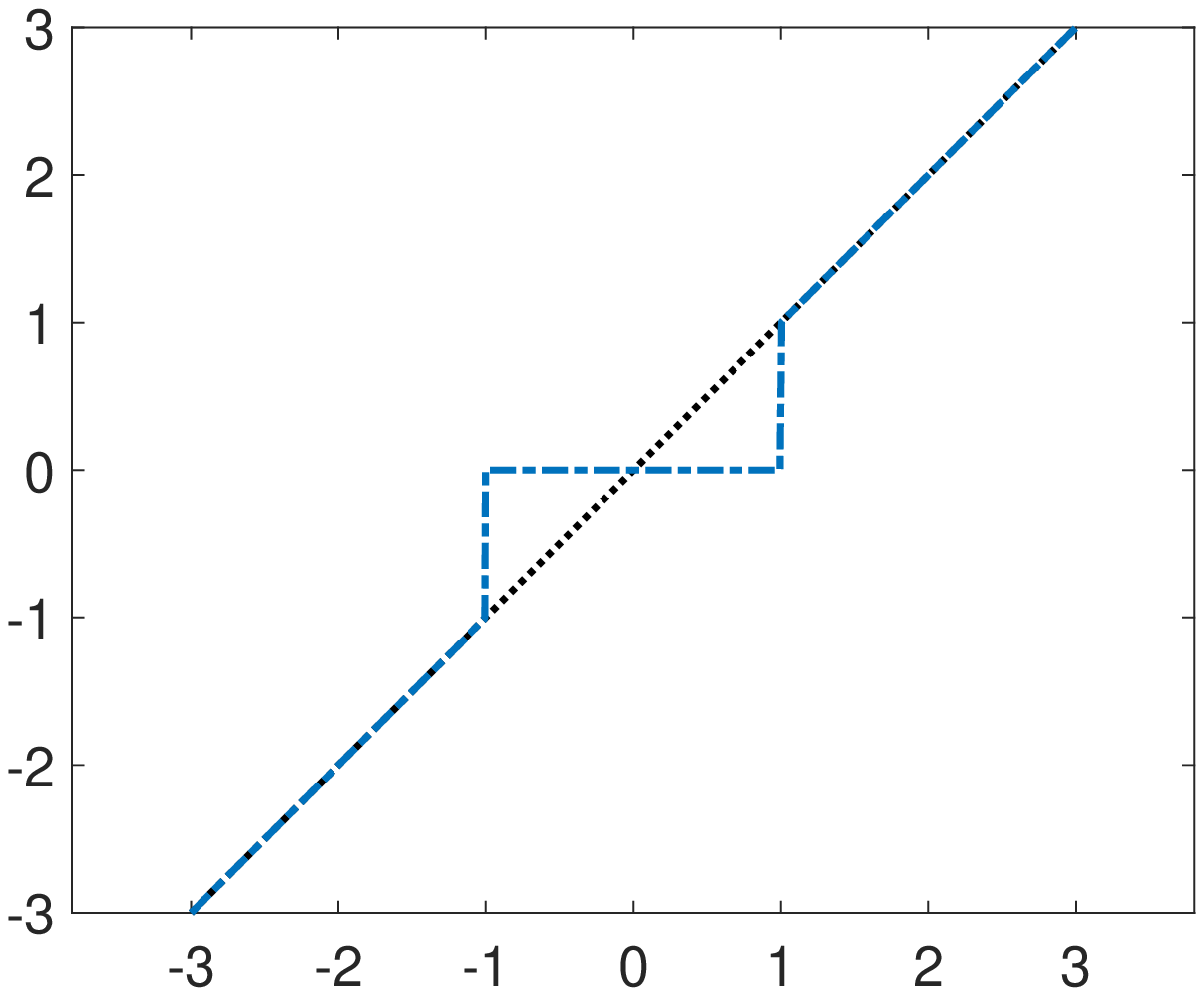}&
	\includegraphics[width=0.33\textwidth]{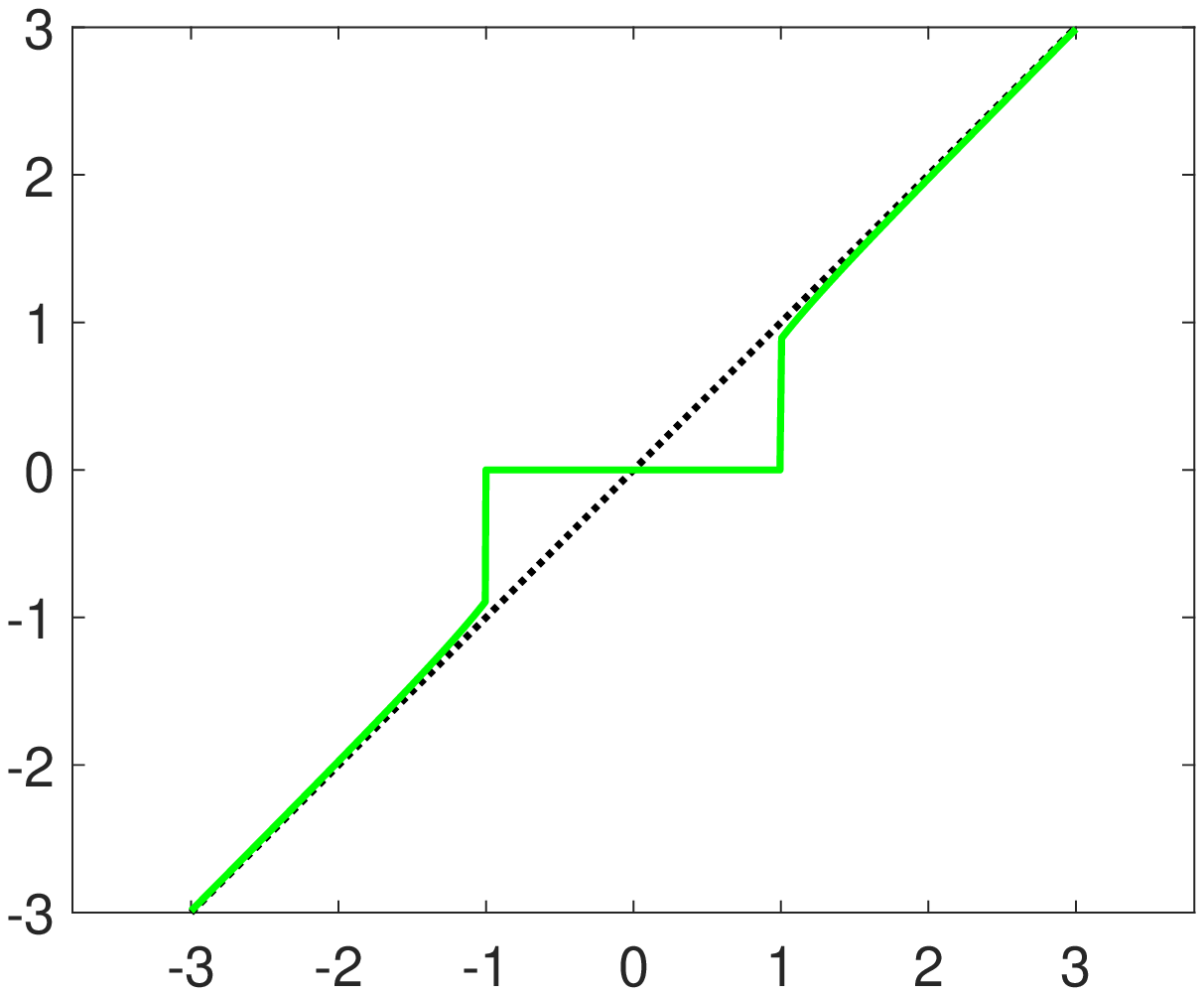}&
	\includegraphics[width=0.33\textwidth]{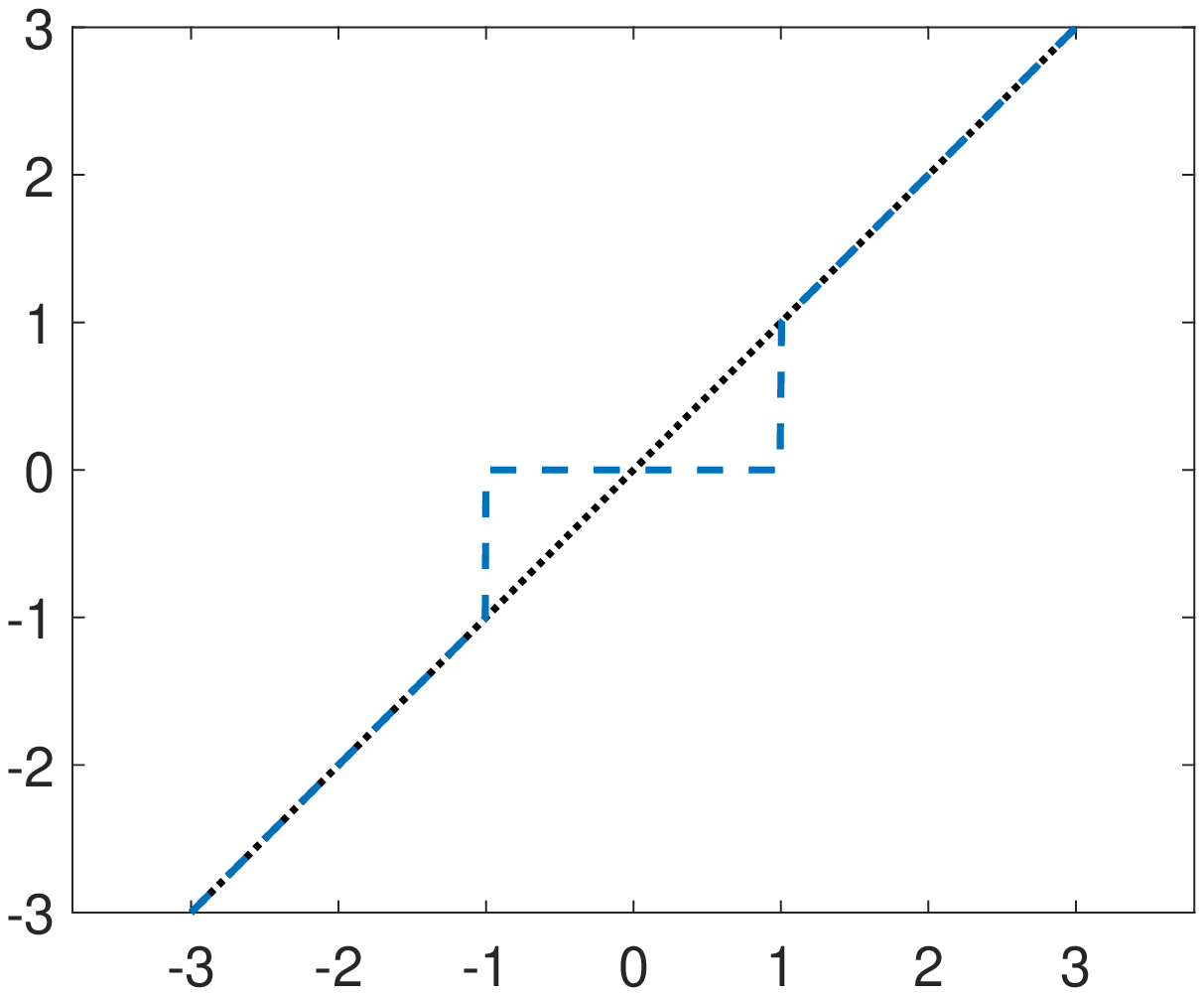}\\
			$L_1$ & TL1, $a=10$ & ERF, $\sigma = 2$\\
	\includegraphics[width=0.33\textwidth]{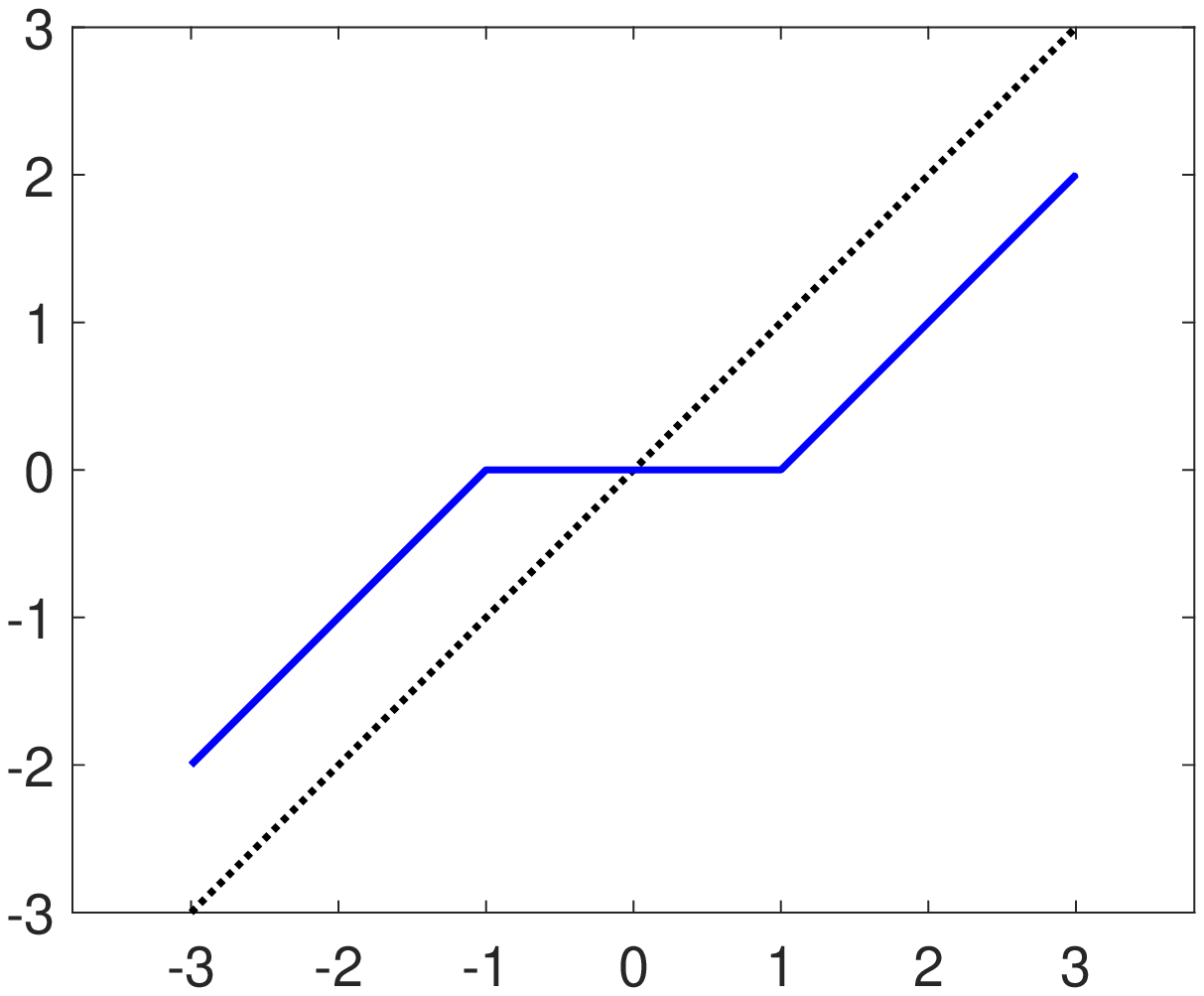}&
	\includegraphics[width=0.33\textwidth]{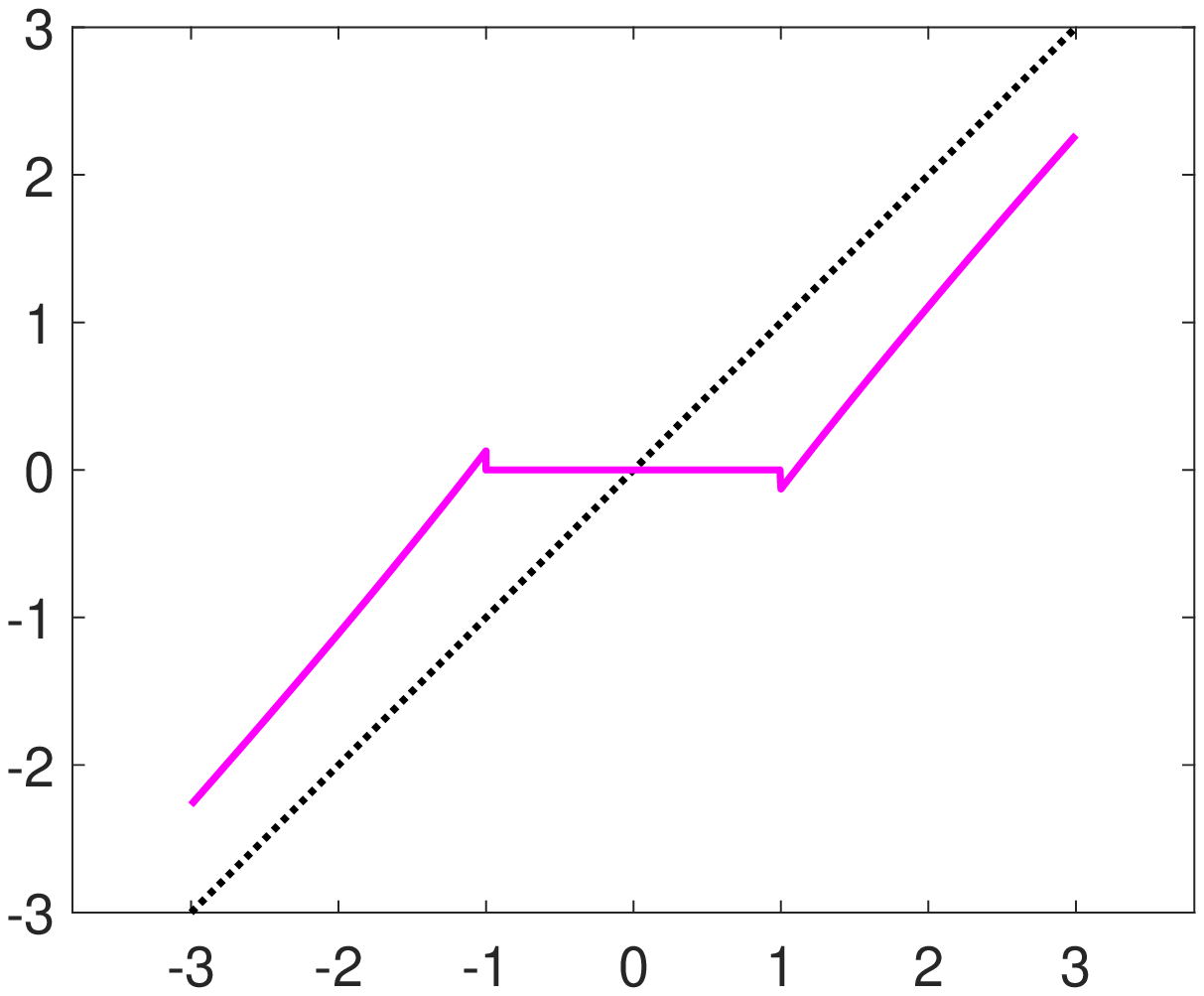}&
	\includegraphics[width=0.33\textwidth]{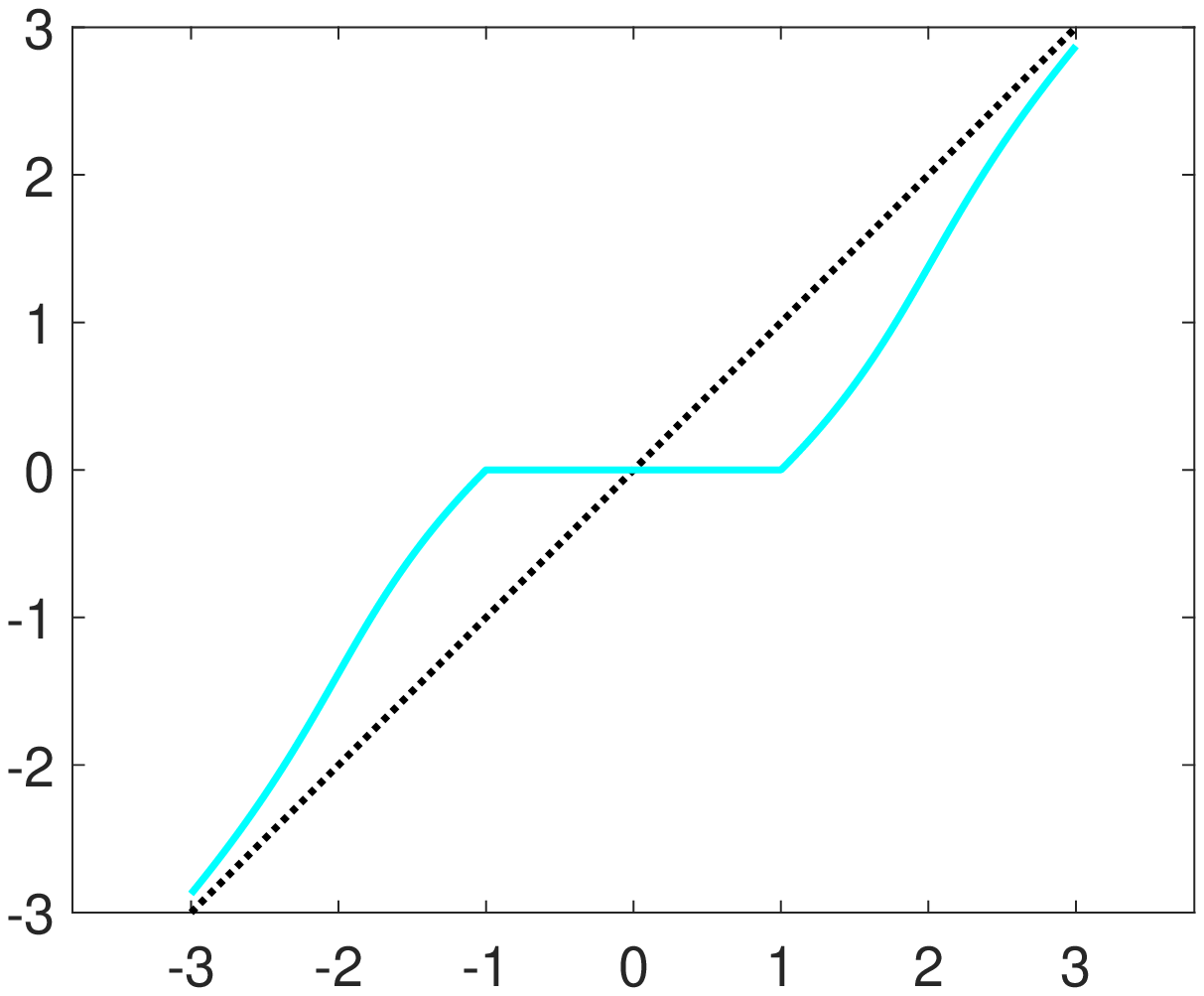}\\
	\end{tabular}
	\caption{Proximal operators of  various sparsity promoting models with $\mu = 1$. } \label{fig:proximal}
\end{figure}

\subsection{Exact recovery guarantee}
Based on the subadditive property, we analyze a generalized null space property (gNSP) \cite{tran2017unified} that guarantees the proposed ERF model exactly finds the desired sparse solution.

\begin{definition}\label{def:NSP}
A matrix $\vA\in\mathbb{R}^{m\times n}$ is said to satisfy a generalized null space property (gNSP) relative to $J_\sigma$ and $S\subseteq\{1,\cdots, n\}$ if
\begin{equation}\label{eq:NSP}
J_\sigma(\vv_S)<J_\sigma(\vv_{S^c})
\end{equation}
for all $\vv\in\ker(\vA)\backslash\{\mathbf{0}\}$. It is said to satisfy the null space property of order $s\leq n$ relative to $J_\sigma$ if it satisfies the null space property relative to $J_\sigma$ and any set $S\subseteq\{1,\cdots, n\}$ with $|S|\leq s$.
\end{definition}

If we replace $J_\sigma$ in \eqref{eq:NSP} by the $L_1$ norm, then Definition~\ref{def:NSP} becomes the standard NSP for exact $L_1$ recovery \cite{donoho2001uncertainty}.

\begin{theorem}
Given a matrix $\vA\in\mathbb{R}^{m\times n}$ and $\sigma>0$, every vector $\vx\in\mathbb{R}^n$ supported on a set $S\subseteq\{1,\cdots, n\}$ is the unique solution of the problem
\[
\min_{\vz}J_\sigma (\vz)\quad\st\quad \vA\vz=\vb
\]
with $\vb=\vA\vx$ if and only if $A$ satisfies  gNSP relative to $J_\sigma$ and $S$.
\end{theorem}

\begin{proof}
Given a fixed index set $S$, let us first assume that every vector $\vx\in\mathbb{R}^n$ supported on $S$ is the unique minimizer of $\min_{\vz}J_\sigma(\vz)$ subject to $\vA\vz=\vA\vx$. Thus for any $\vv\in\ker(\vA)\backslash \{\mathbf{0}\}$, the vector $\vv_S$ is the unique minimizer of $J_\sigma(\vz)$ subject to $\vA\vz=\vA\vv_S$. Since we have
\[
\vA(-\vv_{S^c})=\vA\vv_{S}
\]
and $-\vv_{S^c}\neq\vv_S$, we can get the inequality
\[
J_{\sigma}(\vv_S)<J_\sigma(\vv_{S^c}),
\]
which establishes the null space property relative to $J_\sigma$ and $S$.

Conversely, assume that the null space property relative to $J_\sigma$ and $S$ holds. For $\vx\in\mathbb{R}^n$ supported on $S$ and a vector $\vz\in\mathbb{R}^n$ with $\vz\neq \vx$ and $\vA\vz=\vA\vx$, we have $\supp(\vx-\vz_S)=S$. Then the vector $\vv=\vx-\vz\in\ker(\vA)\backslash \{\mathbf{0}\}$ satisfies that $\vv_S=\vx-\vz_S$. Furthermore, due to the subadditive property, we obtain
\[\begin{aligned}
J_\sigma(\vx)&\leq J_\sigma(\vx-\vz_S)+J_\sigma(\vz_S)\\
&=J_\sigma(\vv_S)+J_\sigma(\vz_S)\\
&<J_\sigma(\vv_{S^c})+J_\sigma(\vz_S)\\
&=J_\sigma(-\vz_{S^c})+J_\sigma(\vz_S)\\
&=J_\sigma(\vz_{S^c})+J_\sigma(\vz_S)=J_\sigma(\vz),
\end{aligned}
\]
which implies that $\vx$ is the unique minimizer of $J_\sigma(\vz)$ subject to the constraint $\vA\vz=\vA\vx$.\qed
\end{proof}

By varying the support, we can get the following theorem for necessary and sufficient conditions of exact sparse recovery. 
\begin{theorem}\label{thm:gNSP}
Given a matrix $\vA\in\mathbb{R}^{m\times n}$ and $\sigma>0$, every $s$-sparse vector $\vx\in\mathbb{R}^n$ is the unique solution of the problem
\[
\min_{\vz}J_\sigma (\vz)\quad\st\quad \vA\vz=\vb
\]
with $\vb=\vA\vx$ if and only if $\vA$ satisfies the null space property of order $s$ relative to $J_\sigma$.
\end{theorem}
Due to the involvement of every $s$-sparse vector in Theorem~\ref{thm:gNSP}, it is NP-hard to verify whether a matrix satisfies gNSP or not.
On the other hand, if we relax ``every $s$-sparse vector,'' then gNSP is no longer necessary.

\section{Algorithms}\label{sect:alg}

We apply the reweighted $L_1$ approach \cite{candes2008enhancing} to minimize the proposed regularization $J_\sigma$. We shall discuss two optimization formulations: constrained and unconstrained, separately.

\subsection{Constrained formulation}\label{sect:alg_con}
Consider an ERF-regularized minimization problem with a linear constraint
\begin{equation}\label{eqn:con}
\min_{\h x\in\R^n} J_\sigma(\h x) \quad \mbox{s.t.} \quad \vA\h x = \h b.
\end{equation}
According to the general IRL1 framework \eqref{eqn:con_rwl14general}, the objective function in \eqref{eqn:con} can be expressed as $J_\sigma (\h x)=F_2(G(\h x)),$ where $F_2 =\Phi^{\mathrm{ERF}}_\sigma$ and $G(\h x) = |\h x|.$ By calculating the derivative of the error function \eqref{prop_erf},
 we obtain the following iterative scheme
\begin{equation}\label{eqn:con_rwl1}
\left\{\begin{array}{l}
w_j^{k} = \exp\{-(\frac{x_j^k}{\sigma})^2\}\\
\h x ^{k+1} = \arg\min_{\h x\in\R^n} \sum_{j=1}^n w^{k}_j|x_j| \quad \mbox{s.t.} \quad \vA\h x = \h b\\
\end{array}
\right.
\end{equation}
The $\h x$-subproblem in \eqref{eqn:con_rwl1} can be cast as a linear programming. We use the commercial Gurobi solver (\url{https://www.gurobi.com/}) to solve this subproblem.  Convergence of the scheme \eqref{eqn:con_rwl1} is presented in Theorem~\ref{thm:con_convg}.

\begin{theorem}\label{thm:con_convg}
	The sequence $\{\vx^k\}_{k=1}^\infty$ generated by the reweighted $L_1$ iteration \eqref{eqn:con_rwl1} is bounded. It has a convergent subsequence and any accumulation point of $\{\vx^k\}_{k=1}^\infty$ is a stationary point of~\eqref{eqn:con}.
\end{theorem}

\begin{proof}
We start by showing that the sequence $\{\vx^k\}_{k=1}^\infty$ is bounded. In particular, we aim to show that the sequence of $\{\vx^{k}\}_{k=1}^\infty$ is in the convex hull constructed by the set $\{\vx: \exists S \mbox{ such that } \vA_S \mbox{ has full column rank, } \vA_S\vx_S=\vb, \vx_{S^c}=\vzero\}$, where $S^c$ is the complement of $S$.
Since the number of linear independent submatrices $\vA_S$ is finite, the convex hull is bounded and hence the sequence is bounded.

Let $\bar{\h x}$ be an optimal solution for the $\h x$-subproblem in \eqref{eqn:con_rwl1}, and $S$ be the set of corresponding indices of nonzero components in $\bar{\h x}$. Given an arbitrary set of weights, denoted by $w_j$, we consider the following problem restricted to the index set $S$,
\begin{equation}\label{proof:LP}
\min_{\h x\in\R^n} \sum_{j\in S}w_j |x_j|,\quad\st\quad \vA_S\h x_S=\h b,~\vx_{S^c}=\mathbf{0}.
\end{equation}
The optimality condition of \eqref{proof:LP}  is
\begin{align} \label{eq:opt_c}
\begin{bmatrix}
\vM & \mathbf{0} \\
 \mathbf{0} & \vA_S^\top
\end{bmatrix}
\begin{bmatrix}
\bar\vx\\\vy
\end{bmatrix} :=\begin{bmatrix}
	\vA_S  & \mathbf{0} & \mathbf{0} \\
	\mathbf{0} & \vI& \mathbf{0}\\
	\mathbf{0} & \mathbf{0} & \vA_S^\top
\end{bmatrix}
\begin{bmatrix}
\bar\vx_S\\\bar\vx_{S^c}\\\vy
\end{bmatrix} = \begin{bmatrix}
\vb\\ \mathbf{0} \\ \hat\vw_S
\end{bmatrix},\end{align}
where $\vy$ is the dual variable and $\hat\vw_S=(\vw\odot\sign(\bar \vx))_S$.
If $\ker(\vM)\neq \{\h 0\}$, then we pick a nonzero vector $\bigtriangleup \vx$ from its kernel (note $(\bigtriangleup \vx)_{S^c}=\vzero$).
Thus, from the optimality condition~\eqref{eq:opt_c}, we can find $t_1<0$ and $t_2>0$ such that $\bar\vx+t\bigtriangleup\vx$ is also an optimal solution for any $t\in[t_1,t_2]$. Both $t_1$ and $t_2$ exist, otherwise the optimal objective value will be $+\infty$.
Furthermore, we can choose $t_1$ (and $t_2$) such that at least one component of $\hat\vx_S:= \bar\vx_S+t_1(\bigtriangleup\vx)_S$ (and $\breve\vx_S :=\bar\vx_S+t_2(\bigtriangleup\vx)_S$)  is zero.
It implies that $\bar\vx$ is a weighted average of $\breve\vx$ and $\hat\vx$, both of which have fewer nonzero components. Then we can apply the same technique on $\hat\vx$ and $\breve\vx$ until we end up solutions such that $\ker(\vM)=\{\mathbf{0}\}$. Therefore, we have $\bar\vx$ is within a convex hull constructed by some solutions of $\vA_S\vx_S=\vb$ and $\vx_{S^c}=\vzero$ with linearly independent $\vA_S$. Since the number of submatrices is finite, we know the whole sequence is bounded.

	
Now that $\{\vx^k\}_{k=1}^\infty$ is bounded, then the Bolzano--Weierstrass Theorem guarantees the existence of a convergent subsequence, denoted by $\{\vx^{n_k}\}_{k=1}^\infty$. Assume that it converges to $\vx^*$. Since $\h x^{n_k}\rightarrow \h x^*$ and $\vA\h x^{n_k} = \h b$, we have $\vA\h x^*=\h b.$	
	Due to the $\h x$-subproblem definition in \eqref{eqn:con_rwl1}, it is straightforward to have 
\[
\sum_{j=1}^{n} w_j^k|x_j^{k+1}|\leq \sum_{j=1}^{n} w_j^k|x_j^{k}|.
\]	
As a result,  we have the following inequality
	\begin{align*}
	 J_\sigma(\vx^k)- J_\sigma(\vx^{k+1}) \geq& \sum_{j=1}^n \big[\Phi_\sigma(x_j^k)-\Phi_\sigma(x_j^{k+1})\big]-\sum_{j=1}^nw_j^{k}\big(|x_j^k|-|x_j^{k+1}|\big)\nonumber\\
	=& \sum_{j=1}^n\Big[\Phi_\sigma(x_j^k)-\Phi_\sigma(x_j^{k+1})-w_j^{k}\big(|x_j^k|-|x_j^{k+1}|\big)\Big]\geq 0,
	\end{align*}
since $\Phi_\sigma(\cdot)$ is a concave function and $w_j^k$ is the derivative of $\Phi_\sigma$ evaluated at $|x_i^k|$.
We have $\vx^k-\vx^{k+1}\rightarrow\vzero$ and $\vx^{n_k+1}$ converges to $\vx^*$.
According to the optimality condition of~\eqref{eqn:con_rwl1}, we have $p_j^{n_k+1}\in\partial |x_j^{n_k+1}|$ and
	$\h w^{n_k+1}\odot \h p^{n_k+1}$ is in the range of $\vA^\top$.
Since the sequence $\{p_j^{n_k+1}\}$ is bounded by $\pm 1$, it has a convergent subsequence. Without loss of generality, we assume
 that $\{p_j^{n_k+1}\}$ converges itself. Therefore, we have
	\begin{align*}
	\exp\{-({x_j^*\over \sigma})^2\}p_j^* = \lim_{k\rightarrow \infty}w_j^{n_k}p_j^{n_k+1},
	\end{align*}
which	is in the range of $\vA^\top$ and $p_j^*\in \partial|x_j^*|$.
	This shows that $\vx^*$ is a stationary point of~\eqref{eqn:con}.
\qed
\end{proof}

\subsection{Unconstrained formulation}\label{sect:alg_uncon}
In the noisy case, we consider an unconstrained problem of the form
\begin{equation}\label{eqn:uncon}
\min_{\h x\in\R^n} \lambda J_\sigma(\h x) +\frac 1 2 \|\vA\h x - \h b\|_2^2,
\end{equation}
with a positive parameter $\lambda$. The reweighted $L_1$ algorithm requires to solve the following subproblem:
\begin{equation}\label{eqn:uncon_rwl1_xsub}
\h x ^{k+1} = \arg\min_{\h x\in \R^n} \lambda\sum_{j=1}^n w^{k}_j|x_j|+\frac 1 2 \|\vA\h x - \h b\|_2^2,
\end{equation}
where the weight vector $\h w^{k}$ is defined the same as in \eqref{eqn:con_rwl1}. We establish the convergence of the iterative scheme \eqref{eqn:uncon_rwl1_xsub}, followed by a proposed algorithm for the subproblem.

\begin{theorem}\label{thm:uncon_convg}
The sequence $\{\vx^k\}_{k=1}^\infty$ generated by the reweighted $L_1$ iteration \eqref{eqn:uncon_rwl1_xsub} is bounded and has a convergent subsequence. Any accumulation point of $\{\vx^k\}_{k=1}^\infty$ is a stationary point of~\eqref{eqn:uncon}.
\end{theorem}

\begin{proof}
We first prove the boundedness of $\{\vx^k\}_{k=1}^\infty$ similar to the proof of Theorem~\ref{thm:con_convg}.
Let $\bar{\h x}$ be an optimal solution for the $\h x$-subproblem in \eqref{eqn:uncon_rwl1_xsub}, and $S$ be the set of corresponding indices of nonzero components in $\bar{\h x}$. We consider the following problem restricted to the index set $S$,
\begin{equation}\label{proof:bdd_sub_prob}
\min_{\h x}~\lambda\sum_{j\in S} w_j|x_j|+\frac 1 2 \|\vA_S\h x_S - \h b\|_2^2, \quad \st~\quad \vx_{S^c}=\mathbf{0},
\end{equation}
of which $\bar\vx$ is an optimal solution. The optimality condition is
\begin{align} \label{eq:opt_uc}
\vM	\vx:=\begin{bmatrix}
\vA_S^\top \vA_S  & \mathbf{0} \\
\mathbf{0} & \vI
\end{bmatrix}
\begin{bmatrix}
\bar\vx_S\\\bar\vx_{S^c}
\end{bmatrix} = \begin{bmatrix}
\vA_S^\top\vb-\lambda\hat\vw_S\\ \mathbf{0}
\end{bmatrix},\end{align}
where $\hat\vw_S=(\vw\odot\sign(\bar \vx))_S$.
Then using the same technique in the proof of Theorem~\ref{thm:con_convg}, we can show that $\{\vx^k\}_{k=1}^\infty$ is within a convex hull constructed by $\{\vx:\exists S \mbox{ and }\hat\vw\in(0,1]^n\mbox{ such that } \vx_S=(\vA_S^\top \vA_S)^{-1}(\vA_S^\top\vb-\lambda \hat\vw_S)\}$. This set is bounded, so the optimal solution $\bar\vx$ is also bounded.

	
Because $\{\vx^k\}_{k=1}^\infty$ is bounded, there exists a subsequence $\{\vx^{n_k}\}$ convergent to $\vx^*$. It follows from the optimality condition of~\eqref{eqn:uncon_rwl1_xsub} that
\begin{align}\label{eqn:opt_con}
0 =\lambda w_j^{k}p_j^{k+1} + \vA_j^\top(\vA\vx^{k+1}-\vb),
\end{align}
where $p_j^{k+1}\in\partial |x_j^{k+1}|$ and $\vA_j$ is the $j$-th column of $\vA$.
Then we have
	\begin{align*}
	&{1\over2}\|\vA\vx^k-\vb\|_2^2-{1\over2}\|\vA\vx^{k+1}-\vb\|_2^2 \nonumber\\
	=& {1\over2}\|\vA\vx^k-\vA\vx^{k+1}\|_2^2+\langle \vx^k-\vx^{k+1},\vA^\top(\vA\vx^{k+1}-\vb)\rangle\nonumber\\
	=&{1\over2}\|\vA\vx^k-\vA\vx^{k+1}\|_2^2-\lambda\sum_{j=1}^n\langle x_j^k-x_j^{k+1},w_j^{k}p_j^{k+1}\rangle\nonumber\\
	\geq& {1\over2}\|\vA\vx^k-\vA\vx^{k+1}\|_2^2-\lambda\sum_{j=1}^nw_j^{k}(|x_j^k|-|x_j^{k+1}|),
	\end{align*}
where the last inequality is guaranteed by the subgradient property. We further obtain that
\begin{align*}
	& \left(\lambda J_\sigma(\vx^k)+{1\over2}\|\vA\vx^k-\vb\|_2^2\right)-\left(\lambda J_\sigma(\vx^{k+1})+{1\over2}\|\vA\vx^{k+1}-\vb\|_2^2\right) \nonumber\\
	=& \lambda\sum_{j=1}^n\left[\Phi_\sigma(x_j^k)-\Phi_\sigma(x_j^{k+1})-w_j^{k}(|x_j^k|-|x_j^{k+1}|)\right]+{1\over2}\|\vA\vx^k-\vA\vx^{k+1}\|_2^2\geq 0.
	\end{align*}
We have $\vx^k-\vx^{k+1}\rightarrow\vzero$ as $k\to\infty$ and $\vx^{n_k+1}\to\vx^*$ as $k\to\infty$.  
Since the sequence $\{p_j^{n_k+1}\}$ is bounded by $\pm 1$, it has a convergent subsequence. Without loss of generality, we assume that $\{p_j^{n_k+1}\}$ converges itself. Thus
\begin{align*}
0 = &~\lim_{k\rightarrow \infty}\lambda w_j^{n_k}p_j^{n_k+1}+\vA_j^\top(\vA\vx^{n_k+1}-\vb)\\
  = &~\lambda e^{-{(x_j^*)^2\over \sigma^2}}p_j^*+\vA_j^\top(\vA\vx^*-\vb),
\end{align*}
where $p_j^*\in\partial |x_j^*|$.
Hence, $\vx^*$ is a stationary point of~\eqref{eqn:uncon}.
\qed
\end{proof}


We apply the ADMM to solve~\eqref{eqn:uncon_rwl1_xsub} by introducing an auxiliary variable $\h y$ and splitting the objective function as
\begin{equation}
\min_{\h x, \h y\in \R^n} \lambda\sum_{j=1}^n w_j|x_j|+\frac 1 2 \|\vA\h y - \h b\|_2^2 \quad \st \quad \h x = \h y. \\
\end{equation}
We omit the (outer) iteration index $k$ when the context is clear. The corresponding augmented Lagrangian can be expressed by
\begin{equation}
\mathcal L(\h x, \h y; \h u) := \lambda\sum_{j=1}^n w_j|x_j|+\frac 1 2 \|\vA\h y - \h b\|_2^2 + \frac \delta 2 \|\h x - \h y + \h u\|_2^2, \\
\end{equation}
where $\h u$ is the dual variable and $\delta$ is a positive parameter. The ADMM algorithm involves the following steps:
\begin{eqnarray}\label{amp}
\left\{
\begin{array}{lcl}
\h x_{l+1}&=&\arg\min\limits_{\h x} \mathcal {L}(\h x,\h y_{l};\h u_{l}),\\
\h y_{l+1}&=&\arg\min\limits_{\h y} \mathcal {L}(\h x_{l+1},\h y;\h u_{l}),\\
\h u_{l+1}&=&\h u_{l} + \h x_{l+1} - \h y_{l+1},
\end{array}
\right.
\end{eqnarray}
where the inner iteration is indexed by $l.$ There are closed-form solutions for both subproblems of $\h x$ and $\h y$ given by
\begin{align}
&\h x_{l+1} = \mathrm{shrink} (\h y_l-\h u_l,\frac{\lambda}{\delta} \h w),\\
&\h y_{l+1} = (\vA^\top \vA+\delta I_d)^{-1}(\vA^\top \h b+\delta \h x_{l+1}+\delta \h u_l),
\end{align}
where $I_d$ denotes the identity matrix. The overall algorithm for solving the unconstrained ERF-regularized model is summarized in Algorithm~\ref{alg:uncon}.

\begin{algorithm}[!ht]
	\caption{The iterative reweighted $L_1$ algorithm for solving the unconstrained ERF-regularized model~\eqref{eqn:uncon}. }\label{alg:uncon}
	\begin{algorithmic}[1]
		\STATE{{\bf Input:} $\vA\in \mathbb{R}^{m\times n}, \h b\in \mathbb{R}^{m}$, $ \sigma, \lambda, \delta >0,$ and MaxOuter/MaxInner.}
		\STATE{{\bf Initialization:} $k=1$ and solve for the $L_1$ minimization to get $\h x^1$.}
		\WHILE{$k < $ MaxOuter or other stopping criteria}
		\STATE{$\h w^{k} = \exp\{-(\frac{\h x^k}{\sigma})^2\}$}
		\STATE{ $l = 1, \h y_l=\h x^k, \h u_l = \h 0.$}
		\WHILE{$l < $ MaxInner or other stopping criteria}
		\STATE{$\h x_{l+1} = \mathrm{shrink} (\h y_l-\h u_l,\frac{\lambda}{\delta} \h w^k).$ }
		\STATE{$\h y_{l+1} = (\vA^\top \vA+\delta I_d)^{-1}(\vA^\top \h b+\delta \h x+\delta \h u).$}
		\STATE{$\h u_{l+1}=\h u_{l} + \h x_{l+1} - \h y_{l+1}.$}
		\STATE{$l \leftarrow l+1$.}
		\ENDWHILE
		\STATE{$\h x^{k+1} = \h x_l,\,k \leftarrow k+1$.}
		\ENDWHILE
		\RETURN $\h x^{k}$
	\end{algorithmic}
\end{algorithm}

\section{Experiments}\label{sect:exp}

In this section, we demonstrate the performance of the proposed algorithms  in comparison to the state-of-the-art methods in sparse recovery.
All the numerical experiments are conducted on a Windows desktop with CPU (Intel i7-6700, 3.19GHz) and $\mathrm{MATLAB \ (R2019a)}$. The codes including test data for the experiments will be available when it is published.

\subsection{Noise-free case}
We focus on one type of sparse recovery problems that involves highly coherent matrices, where the standard $L_1$ model does not work well. Following the works of \cite{DCT2012coherence,louYHX14,yinLHX14}, we consider an over-sampled discrete cosine transform (DCT), defined as $\vA= [\h a_1, \h a_2, \cdots, \h a_n]\in \R^{m\times n}$ with
\begin{equation}\label{eq:oversampledDCT}
\h a_j := \frac{1}{\sqrt{m}}\cos \left(\frac{2\pi  j \h w}{F} \right), \quad j = 1, \cdots, n,
\end{equation}
where $\h w$ is a random vector uniformly distributed in $[0, 1]^m$ and $F\in \R$ is a positive parameter to control the coherence in a way that a larger  value of $F$ yields a more coherent matrix. Throughout the experiments, we consider  over-sampled DCT matrices of size $64\times 1024$. The ground truth $\h x\in \R^{n}$ is simulated as an $s$-sparse signal, where  $s$ is the number of nonzero entries. As suggested in  \cite{DCT2012coherence}, we require a minimum separation of $2F$ in the support of $\h x$. The values of non-zero elements follow Gaussian normal distribution i.e., $ (\h x_s)_i \sim \mathcal{N}(0,1), \ i = 1,2,\dots, s.$

We evaluate the performance of  sparse recovery in terms of \textit{success rate}, defined as the number of successful trials over the total number of trials. A  success is declared if the relative error of the reconstructed solution $\h x^\ast$ to the ground truth $\h x$ is less than $10^{-3}$,  i.e., $\|\h x^\ast-  \h x\|_2/\|\h x\|_2 \leq 10^{-3}$.

Figure~\ref{fig:sigma} examines the performance of the ERF regularization with respect to different choices of $\sigma$, which numerically demonstrates that the proposed regularization approaches to the $L_1$ norm for a large value of $\sigma.$ Following from Figure~\ref{fig:sigma}, we choose $\sigma = 0.1, 0.5, 0.5, 1$  for $F=1,5,10,20,$ respectively, and compare the sparse recovery performance among the state-of-the-art methods in Figure~\ref{fig:succRate}. The competing methods are labeled as $L_0$ (IRL1 \cite{candes2008enhancing}), $L_p$ ($p=1/2$ \cite{chartrand07}), TL1 ($a=1$ \cite{zhangX18}), and $L_1$-$L_2$ \cite{louYHX14,yinLHX14}. We observe that the proposed approach is always the best or at least the second best under all coherence and sparsity levels.

	\begin{figure}
	\centering\setlength{\tabcolsep}{2pt}
	\begin{tabular}{cc}
		$F=1$ & $F=5$\\
		\includegraphics[width=0.48\textwidth]{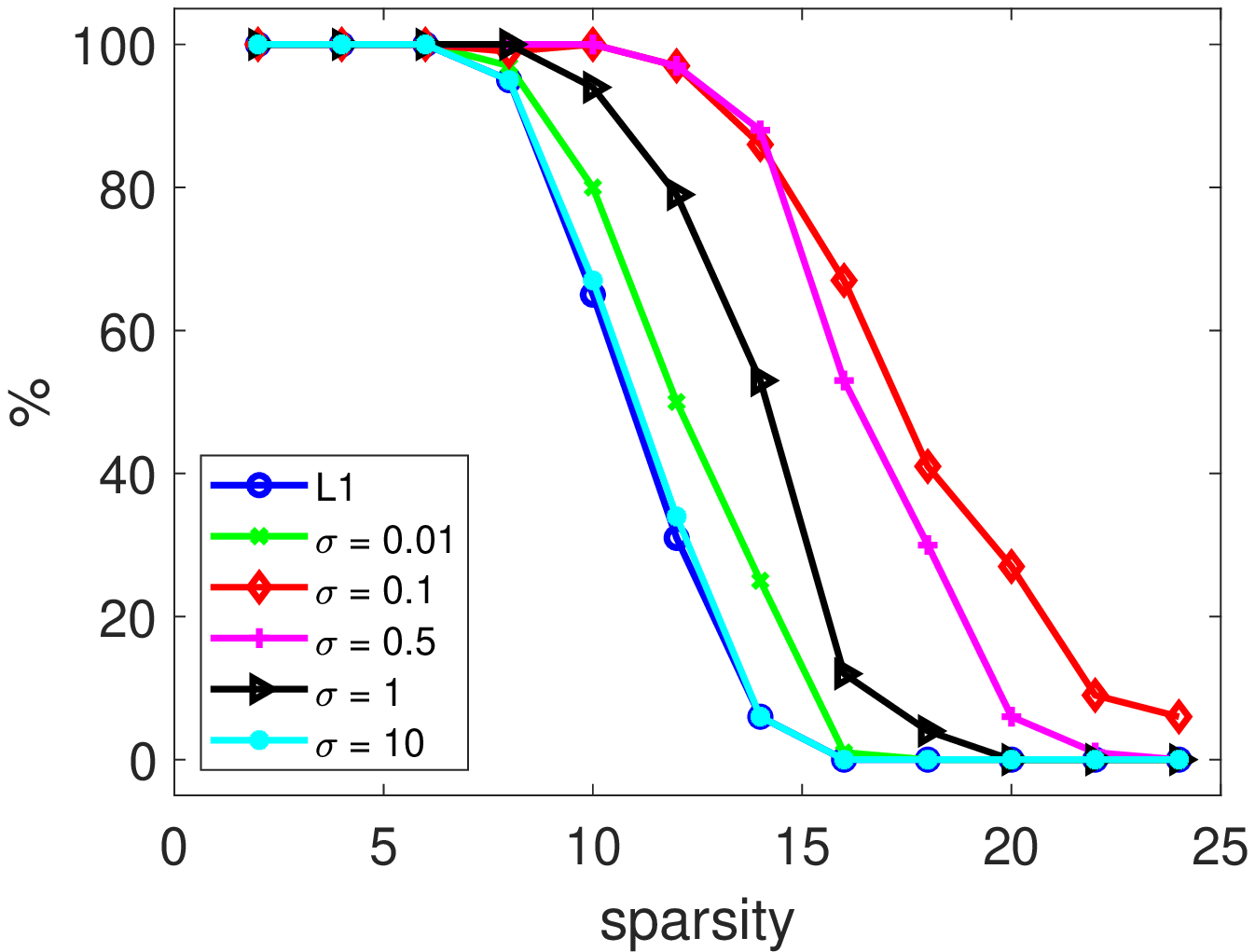} &
		\includegraphics[width=0.48\textwidth]{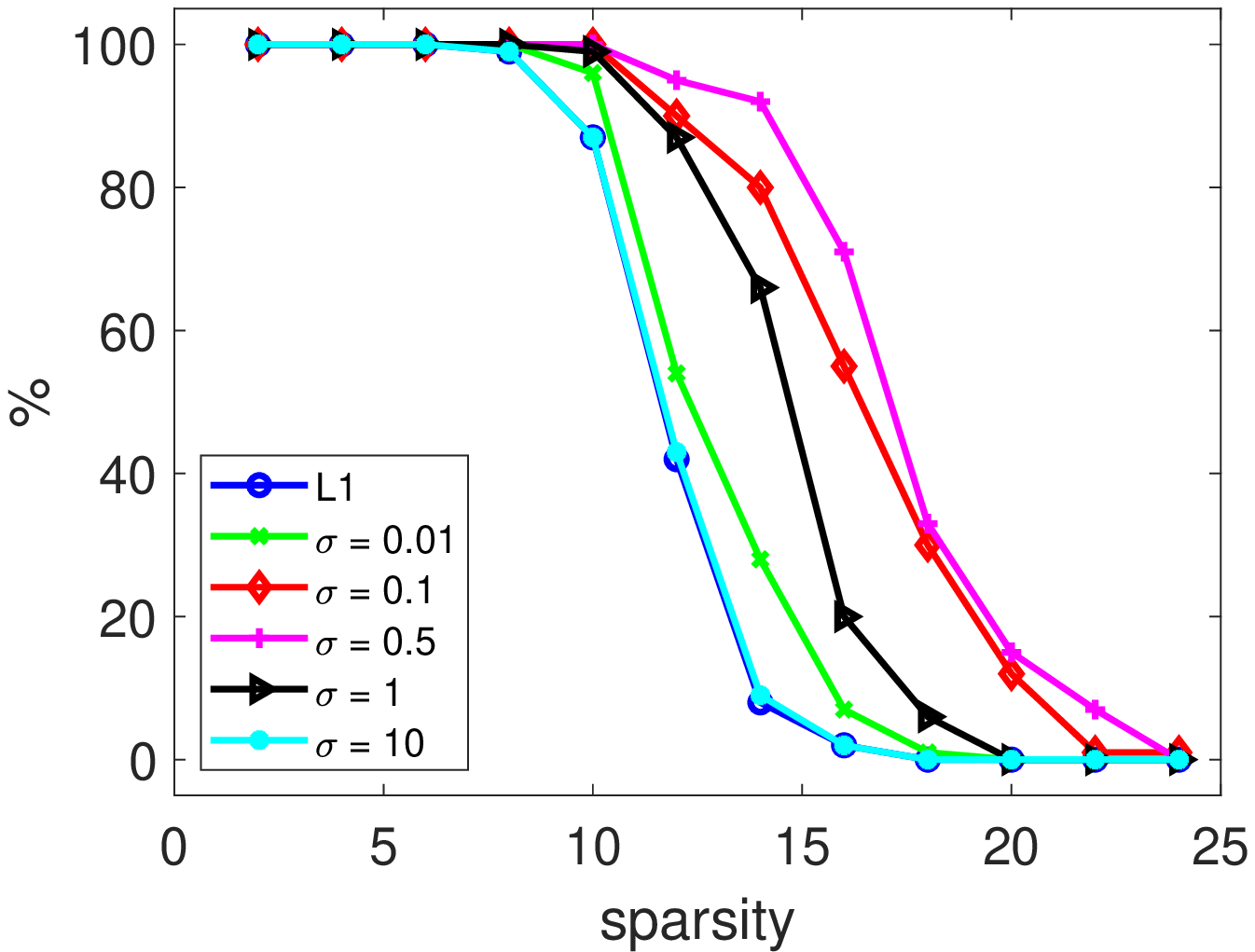}\\
		$F=10$ & $F=20$\\
		\includegraphics[width=0.48\textwidth]{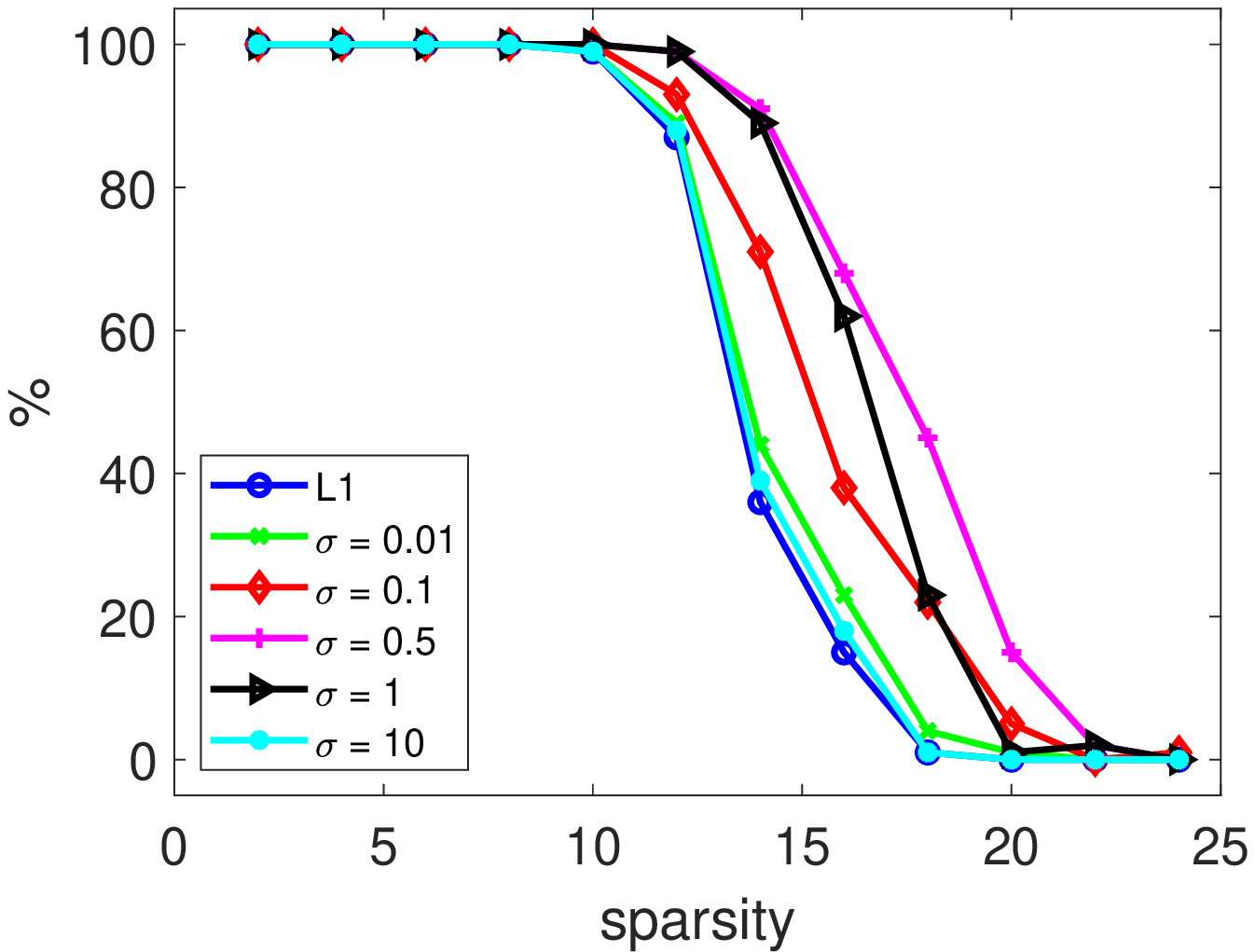} &
		\includegraphics[width=0.48\textwidth]{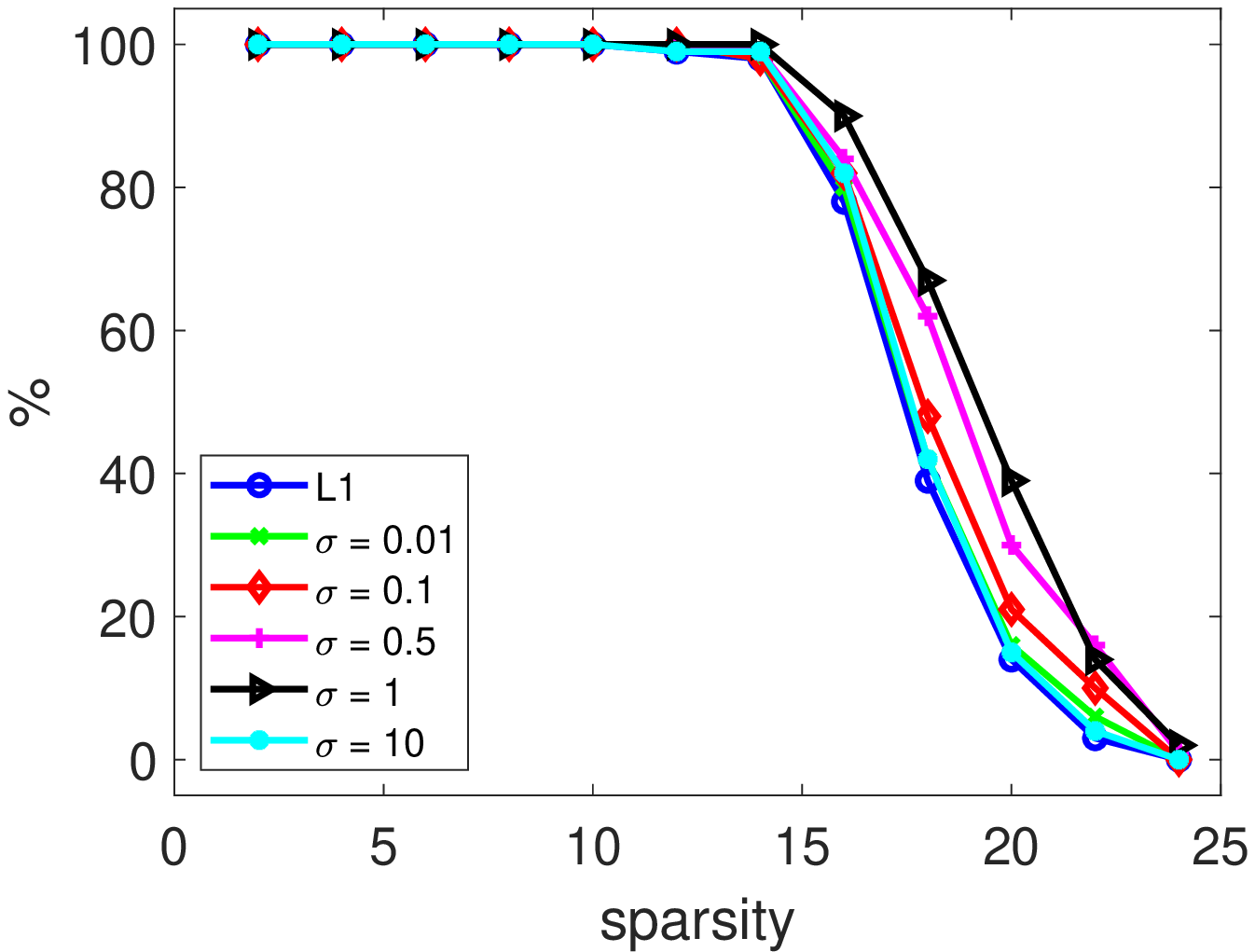}\\
	\end{tabular}
	\caption{The performance of the ERF regularization with respect to the choice of $\sigma$. }
	\label{fig:sigma}
\end{figure}

	\begin{figure}
	\centering\setlength{\tabcolsep}{2pt}
	\begin{tabular}{cc}
		$F=1$ & $F=5$\\
		\includegraphics[width=0.48\textwidth]{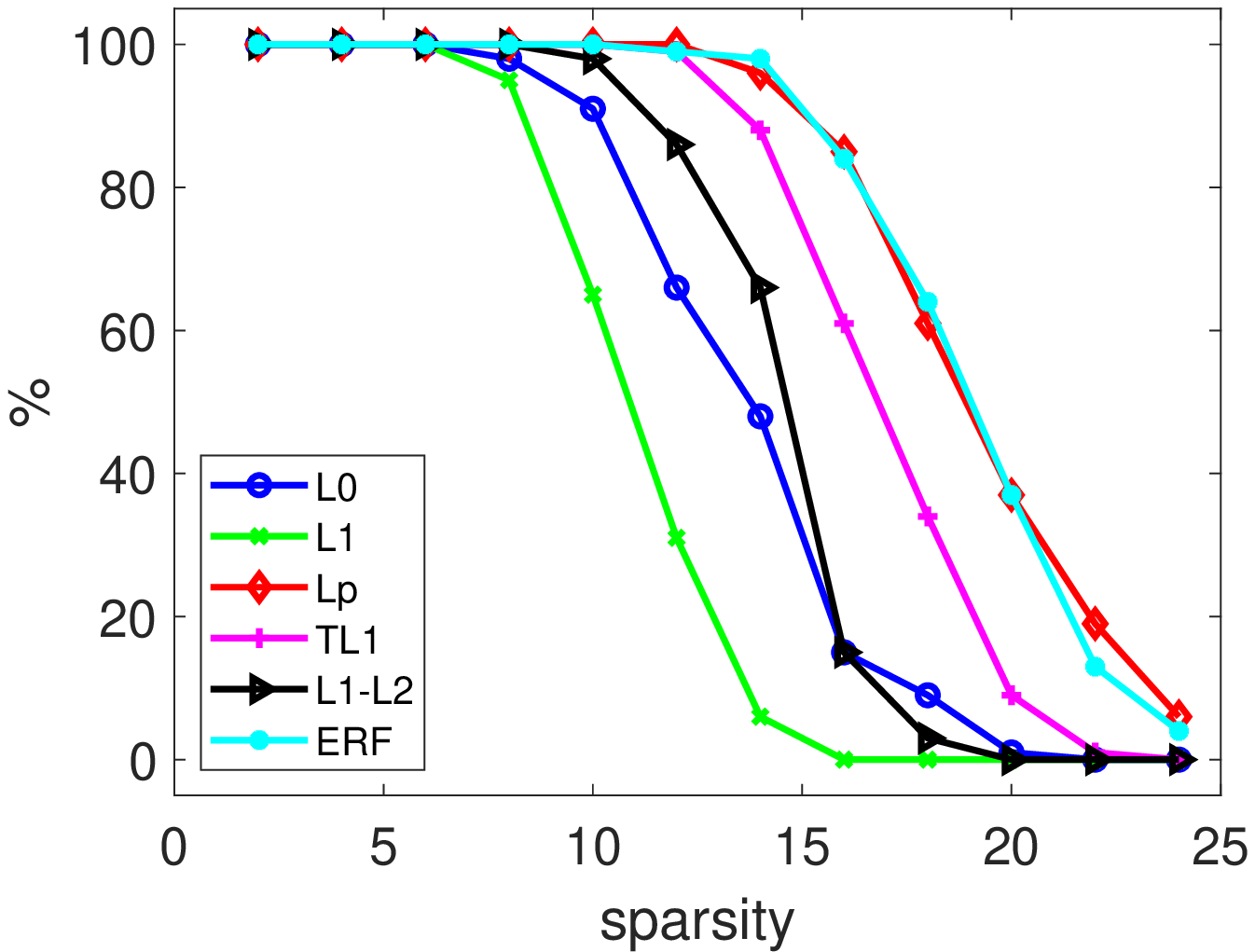} &
		\includegraphics[width=0.48\textwidth]{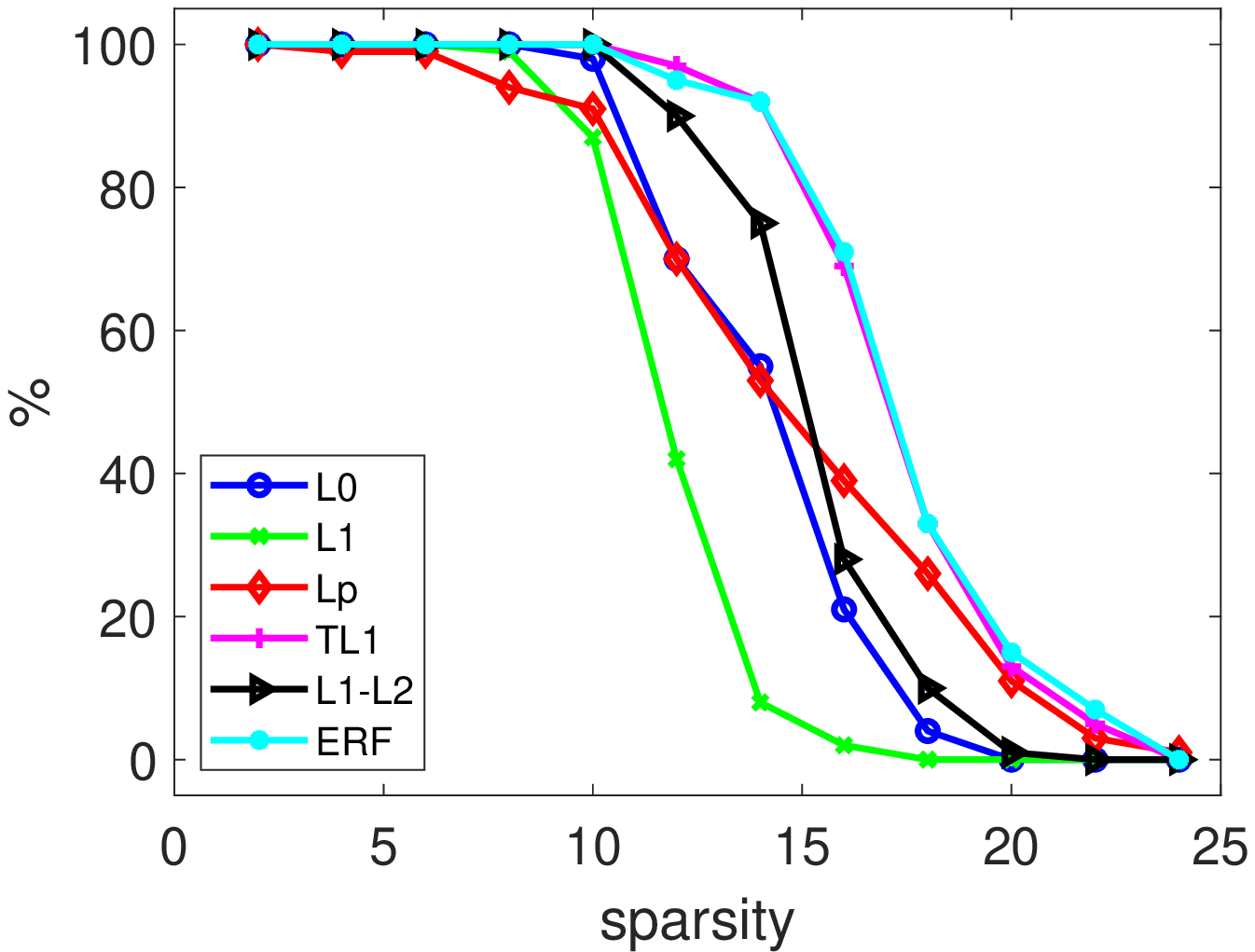}\\
		$F=10$ & $F=20$\\
		\includegraphics[width=0.48\textwidth]{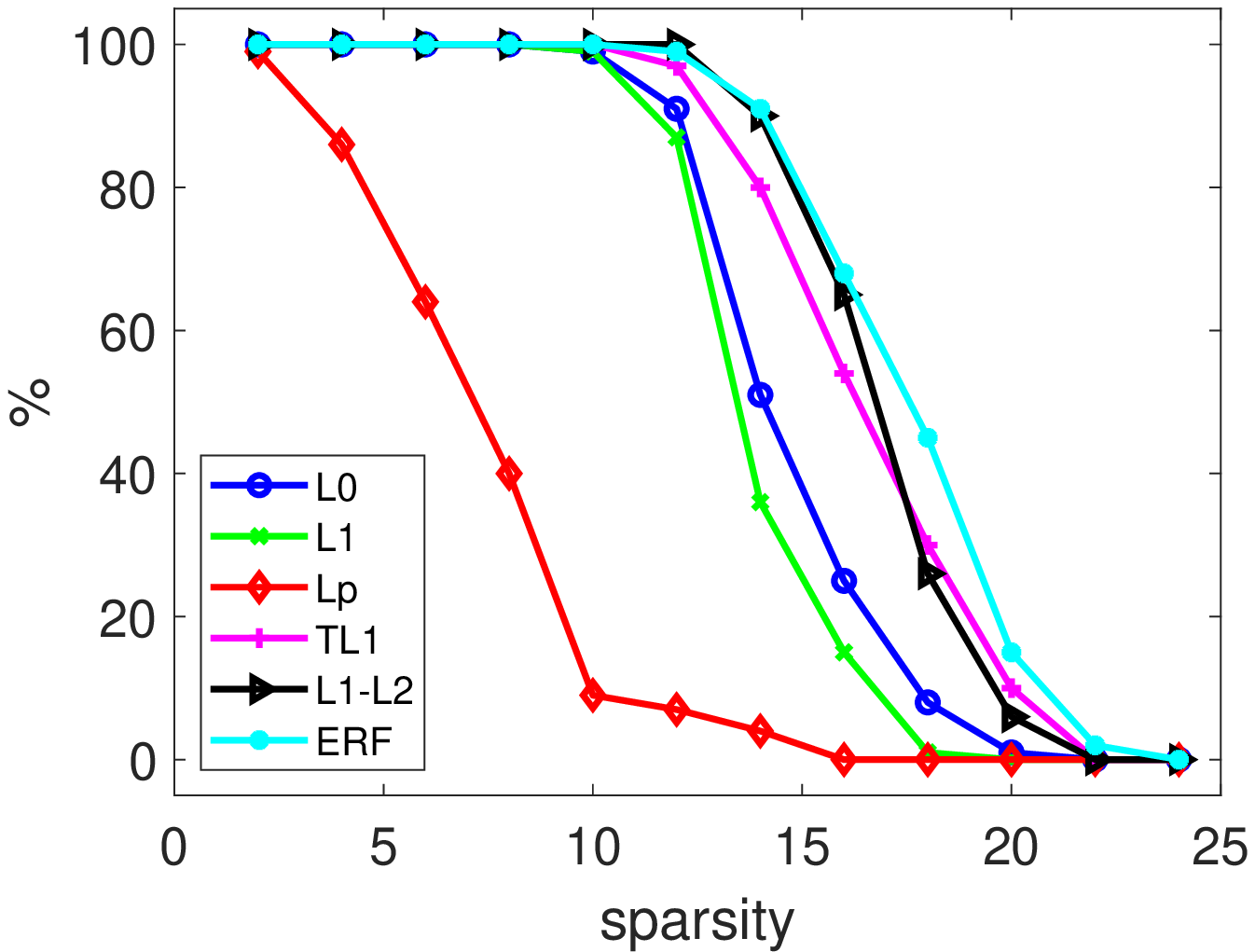} &
		\includegraphics[width=0.48\textwidth]{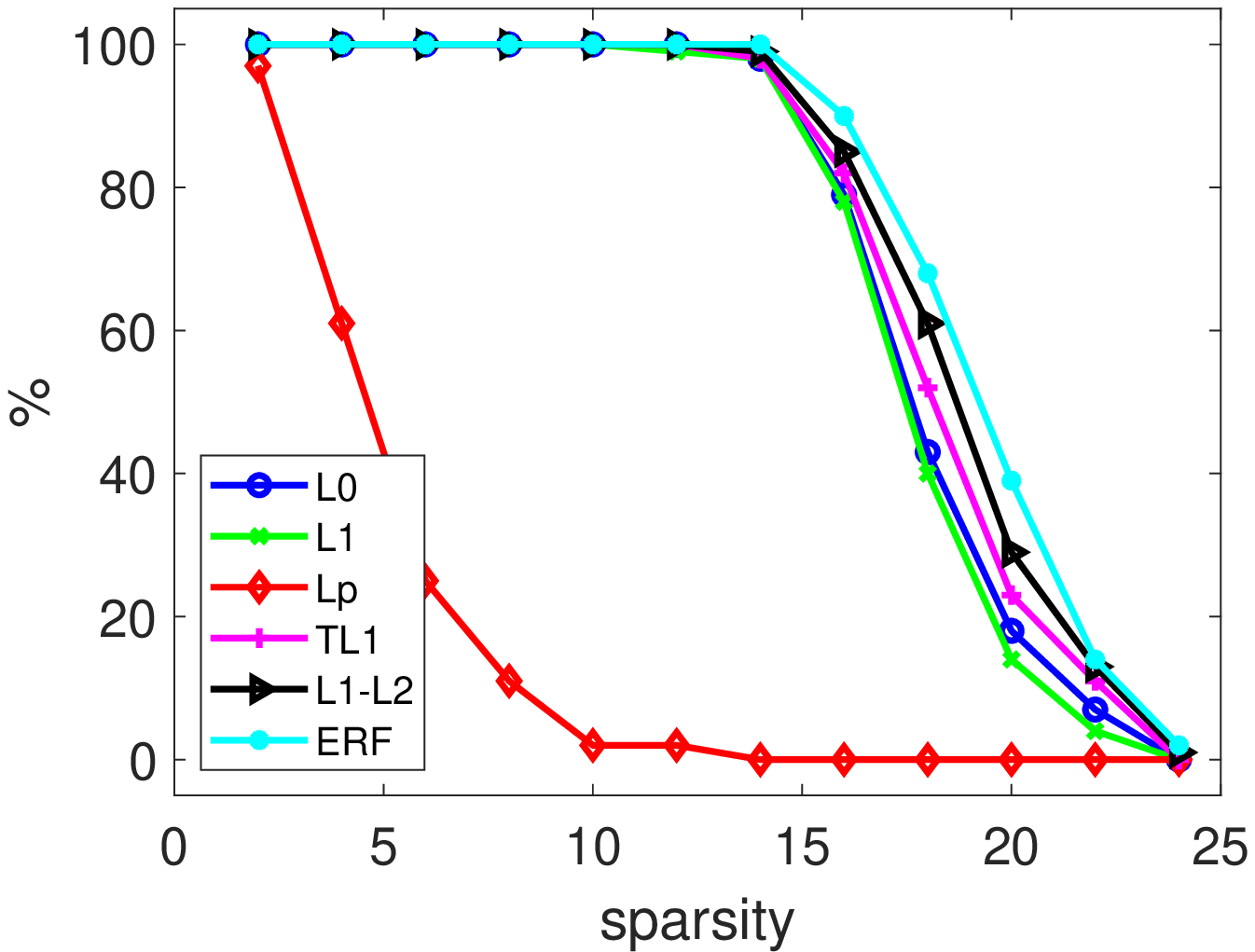}\\
	\end{tabular}
	\caption{The comparison with the state-of-the art methods in sparse recovery. }
	\label{fig:succRate}
\end{figure}

\subsection{Super-resolution}

We also examine the case of super-resolution, in which a coherent sensing matrix is involved.
A mathematical model for super-resolution can be expressed as
\begin{equation}\label{eq:SRdiscrete}
b_k=\frac 1 {\sqrt{N}}\sum_{t=0}^{N-1} x_t e^{-i2\pi kt/N}, \qquad |k|\leq f_c,
\end{equation}
where $i$ is the imaginary unit, $\h x\in\mathbb R^N$ is a vector to be recovered and $\h b\in\mathbb C^n$ is the given low frequency measurements with $n=2f_c+1\ (n<N)$.
This is related to super-resolution in the sense that the underlying signal $\h x$ is defined on a fine grid with spacing $1/N$, while the frequency data of length $n$ imply that one can only expect to recover the signal on a coarser grid with spacing $1/n$. For simplicity, we use matrix notation to rewrite \eqref{eq:SRdiscrete} as $\h b=S_n\mathcal{F} \h x$, where $S_n$ is a sampling matrix that indicates what frequency is collected, $\mathcal F$ is the Fourier transform matrix, and we denote $\mathcal{F}_n=S_n\mathcal{F}$. The frequency cutoff induces a resolution limit inversely proportional to $f_c$; below we set $\lambda_c=1/f_c$, which is referred to as Rayleigh length (a classical resolution limit of hardware \cite{goodman2005introduction}).

We are  interested in reconstructing point sources, i.e.,
$
\h x=\sum_{t_j\in T} c_j\delta_{t_j},
$
where $\delta_{\tau}$ is a Dirac measure at $\tau$, spikes of $\h x$ are located at $t_j$'s belonging to a set $T$, and $c_j$'s are coefficients. Following the work of \cite{candes2014towards}, the sparse spikes are required to be sufficiently separated; please refer to Definition~\ref{def:MS} and Theorem~\ref{thm:MS}.

\begin{definition} {(Minimum Separation)}\label{def:MS}
	Let $\mathbb T$ be the circle obtained by identifying the endpoints on $[0,1]$ and $\mathbb T^d$ the $d$-dimensional torus. For a family of points $T\subset \mathbb T^d$, the minimum separation is defined as the closest warp-around distance between any two elements from $T$,
	\begin{equation}
	\mathrm{MS} :=\triangle (T) := \inf_{(t,t')\in T:t\neq t'} |t-t'|,
	\end{equation}
	where $|t-t'|$ is the $L_\infty$ distance (maximum deviation along any coordinate axis).
\end{definition}

\begin{theorem}{\cite[Corollary 1.4]{candes2014towards}}\label{thm:MS}
	Let $T=\{t_j\}$ be the support of $\h x$. If the minimum distance obeys
	\begin{equation}\label{eq:MSdistance}
	\triangle(T)\geq 2\lambda_c N,
	\end{equation}
	then $\h x$ is the unique solution to $L_1$ minimization:
	\begin{equation}\label{eq:SRviaL1}
	\min \|\h x\|_1 \quad \mbox{s.t.} \quad \mathcal{F}_n\h x=\h b.
	\end{equation}
	If $x$ is real-valued, then the minimum gap can be lowered to $1.87\lambda_c N$.
\end{theorem}

We are interested in the constant in front of $\lambda_c N$ in  \eqref{eq:MSdistance}, referred to as minimum separation factor (MSF). Theorem~\ref{thm:MS} indicates that MSF$\geq 2$  guarantees the exact recovery of $L_1$ minimization. We want to analyze how different sparse recovery algorithms behave with respect to MSF. For this purpose, we consider a sparse signal (ground truth) $\h x_g$  of dimension 1000 with MS = 20.
We vary $f_c$ from 31 to 60, thus MSF$:=\triangle(T) \cdot f_c/N:=$MS$\cdot f_c/N$=$0.62:0.02:1.2$. Denoted $\h x^*$ as the reconstructed signal using any of the methods including $L_1$ via SDP \cite{candes2014towards}, constrained $L_1$-$L_2$ minimization via DCA \cite{louYX16}, and the proposed ERF model via IRL1. we consider 100 random realizations of the same setting to compute the success rates: an incident (or a reconstructed signal $x^*$) is labeled as ``successful'' if $\|\h x^*-\h x_g\|_2/\|\h x_g\|_2<1.5\cdot 10^{-3}$.  Figure~\ref{fig:superRes} shows big advantages of the nonconvex approaches $L_1$-$L_2$ and ERF over the convex $L_1$ approach, while the proposed ERF model is slightly better than $L_1$-$L_2$.

\begin{figure}
	\centering
	\includegraphics[width=0.6\textwidth]{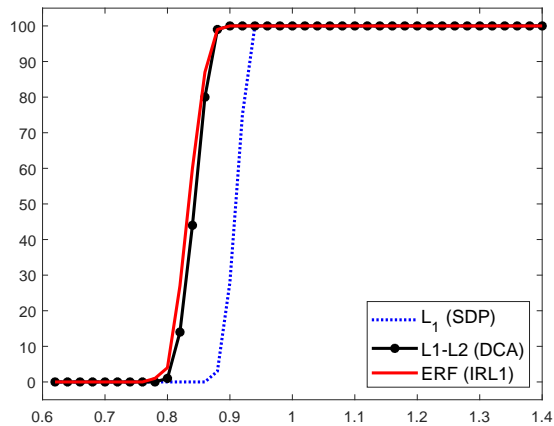}
	\caption{Success rates ($\%$) of fixed MS$=20$ with the ambient dimension $N=1000$.} \label{fig:superRes}
\end{figure}

\subsection{Noisy case}

We provide a series of simulations to demonstrate sparse recovery with noise, following an experimental setup in~\cite{Xu2012}. We consider a signal $\h x$ of length $n=512$ with $s=130$ non-zero elements. We try to recover it from $m$ measurements (denoted by $\h b$) determined by a Gaussian random matrix $\vA$, i.e., a matrix whose columns are normalized with zero-mean and unit Euclidean norm, and Gaussian noise with zero mean and standard deviation $\sigma = 0.1$.  Taking noise into consideration, we use the mean-square-error (MSE) to quantify the recovery performance.
If the support of the ground-truth solution $\h x$ is known, denoted as $\Lambda=\mathrm{supp}(\h x)$, we can compute the MSE of an oracle solution, given by the formula $\sigma^2\mathrm{tr}(\vA_{\Lambda}^T\vA_{\Lambda})^{-1}$, as benchmark.

We compare the proposed ERF model with   $L_{1/2}$  via the half-thresholding method\footnote{We use the author's Matlab implementation with default parameter settings and the same stopping condition adopted as our approach in the comparison.}~\cite{Xu2012}, $L_1$ and $L_1$-$L_2$ (both are solved via ADMM). Each number in Figure~\ref{fig:noisyAll} is based on the average of 100 random realizations under the same setup. When $m$ is small, the sensing matrix becomes coherent, and $L_1$-$L_2$ seems to show advantages and/or robustness over $L_p$ and ERF.  $L_p$ and ERF are asymptotically approaching to the oracle solutions for larger $m$ values.

In Table~\ref{tab:noisy}, we present the mean and standard deviation of MSE and computation time at the four particular $m$ values: 240, 270, 300, and 340. The proposed method achieves the best results, except for larger $m$ value, when the half-thresholding result is the best. But the half-thresholding method is much slower than other competing ones.

\begin{figure}
	\centering
	\includegraphics[width=0.7\textwidth]{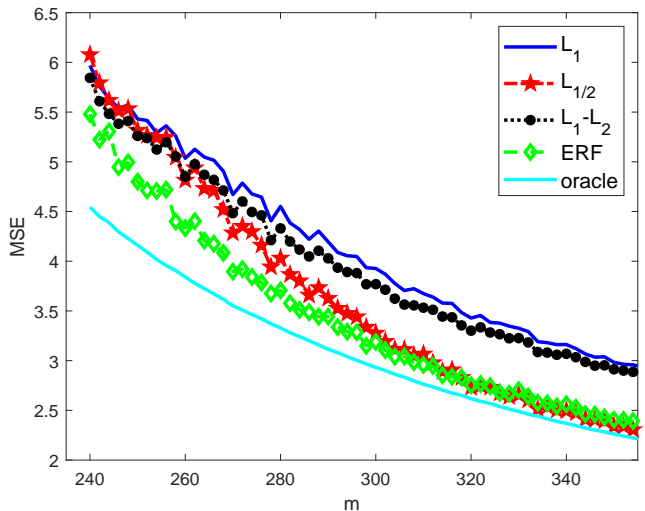}
	\caption{MSE of sparse recovery under the presence of additive Gaussian white noise. The sensing matrix is of size $m\times n$, where $m$ ranges from 240 to 350 and $n = 512$. The ground-truth sparse vector contains 130 non-zero elements. The MSE values are averaged over 100 random realizations.}\label{fig:noisyAll}
\end{figure}

\begin{table}
	\centering
	\begin{tabular}{l||ccc||ccc}\hline\hline
		Methods & $m$ & MSE & Time (sec.) & $m$ & MSE & Time (sec.)  \\\hline
		oracle  &   & \textit{4.54} (0.07) &  &  & \textit{3.55} (0.04) &\\
	$L_{1/2}$~\cite{Xu2012} &  & 6.07 (0.93) & 7.79 (0.92) & &  4.28 (0.66) & 8.40 (2.56)\\
		$L_1$	 & 240&  5.97 (0.75) & 0.19  (0.03) &270 & 4.67 (0.57) &  0.22 (0.03)\\
	$L_1$-$L_2$ &&  5.84 (0.80)& 0.67  (0.06) & & 4.48 (0.58) & 0.81 (0.08) \\
	ERF &  & \textbf{ 5.48} (1.16)& 0.48 (0.08) &&  \textbf{3.90} (0.69) &  0.46 (0.09) \\\hline\hline
		Methods & $m$ & MSE & Time (sec.) & $m$ & MSE & Time (sec.)  \\\hline
oracle  &   & \textit{2.76} (0.02) &  &  & \textit{2.37} (0.02) &\\
$L_{1/2}$~\cite{Xu2012} &  & 3.07 (0.38)&  10.55 (1.91)  & & \textbf{2.50} (0.26) &  11.90 (0.28)\\
$L_1$	 & 310& 3.67 (0.39)& 0.27 (0.05) &340 & 3.16 (0.29) & 0.28 (0.05)\\
$L_1$-$L_2$ && 3.53 (0.37)& 0.96 (0.11) & & 3.07 (0.27)& 1.01 (0.13) \\
ERF &  &  \textbf{2.96} (0.29)& 0.43 (0.06) &&  2.56 (0.24) & 0.40 (0.07)\\\hline
	\end{tabular}
	\caption{Recovery results of noisy signals (mean and standard deviation over 100 realizations). The best results are highlighted in boldface and oracle results are in italics. }\label{tab:noisy}
\end{table}

\section{Conclusions and future works}\label{sect:conclude}
We propose a novel regularization based on the error function for sparse signal recovery. The asymptotic behaviors of the error function indicate that the proposed regularization can approximate the standard $L_0$, $L_1$ norms as the parameter approaches to $0$ and $\infty,$ respectively.
We apply the Newton's method to find a solution for the proximal operator corresponding to the proposed regularization. Plots of asymptotic behaviors and proximal solutions demonstrate that the proposed regularizer is smooth and less biased than the $L_1$ counterpart. We also develop the iterative reweighted algorithms for constrained and unconstrained formulations, both with guaranteed convergence. Experiments demonstrate that the proposed model outperforms the state-of-the-art approaches in sparse recovery in various settings.

Our future work will involve theoretical comparisons between gNSP for the proposed regularizer and NSP for $L_1$. 
In addition, we will develop alternative numerical schemes to minimize the proposed model, e.g., by using the proximal operator.

\begin{acknowledgements}
This research  was initialized at the American Institute of Mathematics Structured Quartet Research Ensembles (SQuaREs), July 22--26, 2019. The authors would like to acknowledge Dr.~Chao Wang for providing sparse recovery codes. WG was partially supported by NSF DMS-1521582. YL was partially supported by NSF CAREER 1846690. JQ was supported by NSF DMS-1941197. 
\end{acknowledgements}

\bibliographystyle{spmpsci}      
\bibliography{refer_l1dl2}

\end{document}

%% file: macros.tex
\usepackage{mathtools}
\usepackage[normalem]{ulem} 

\usepackage{amssymb}
\usepackage{mathrsfs}
\usepackage{algorithmic}
\usepackage{algorithm}

\newcommand{\cut}[1]{{}}

\newcommand{\vb}{{\mathbf{b}}}

\newcommand{\vv}{{\mathbf{v}}}
\newcommand{\vw}{{\mathbf{w}}}
\newcommand{\vx}{{\mathbf{x}}}
\newcommand{\vy}{{\mathbf{y}}}
\newcommand{\vz}{{\mathbf{z}}}

\newcommand{\vA}{{\mathbf{A}}}

\newcommand{\vI}{{\mathbf{I}}}

\newcommand{\vM}{{\mathbf{M}}}

\newcommand{\sign}{\mathrm{sign}}
\newcommand{\vzero}{\mathbf{0}}

\newcommand{\st}{{\text{s.t.}}} 
\newcommand{\tr}{{\mathrm{tr}}} 

\makeatletter
\let\@@span\span
\def\sp@n{\@@span\omit\advance\@multicnt\m@ne}
\makeatother

\newcommand{\bc}{\begin{center}}
\newcommand{\ec}{\end{center}}

\newcommand{\bdm}{\begin{displaymath}}
\newcommand{\edm}{\end{displaymath}}

\newcommand{\beq}{\begin{equation}}
\newcommand{\eeq}{\end{equation}}

\newcommand{\bfl}{\begin{flushleft}}
\newcommand{\efl}{\end{flushleft}}

\newcommand{\bt}{\begin{tabbing}}
\newcommand{\et}{\end{tabbing}}

\newcommand{\beqn}{\begin{align}}
\newcommand{\eeqn}{\end{align}}

\newcommand{\beqs}{\begin{align*}} 
\newcommand{\eeqs}{\end{align*}}  





%% file: main_v4.bbl
\begin{thebibliography}{10}
\providecommand{\url}[1]{{#1}}
\providecommand{\urlprefix}{URL }
\expandafter\ifx\csname urlstyle\endcsname\relax
  \providecommand{\doi}[1]{DOI~\discretionary{}{}{}#1}\else
  \providecommand{\doi}{DOI~\discretionary{}{}{}\begingroup
  \urlstyle{rm}\Url}\fi

\bibitem{bai2018graph}
Bai, Y., Cheung, G., Liu, X., Gao, W.: Graph-based blind image deblurring from
  a single photograph.
\newblock IEEE Trans. Image Process. \textbf{28}(3), 1404--1418 (2018)

\bibitem{candes2014towards}
Cand{\`e}s, E.J., Fernandez-Granda, C.: Towards a mathematical theory of
  super-resolution.
\newblock Comm. Pure Appl. Math. \textbf{67}(6), 906--956 (2014)

\bibitem{CRT}
Cand{\`e}s, E.J., Romberg, J.K., Tao, T.: Stable signal recovery from
  incomplete and inaccurate measurements.
\newblock Comm. Pure Appl. Math \textbf{59}(8), 1207--1223 (2006)

\bibitem{candes2008enhancing}
Cand\'es, E.J., Wakin, M.B., Boyd, S.P.: Enhancing sparsity by reweighted l1
  minimization.
\newblock J Fourier Anal Appl. \textbf{14}(5-6), 877--905 (2008)

\bibitem{chartrand07}
Chartrand, R.: Exact reconstruction of sparse signals via nonconvex
  minimization.
\newblock IEEE Signal Process Lett. \textbf{14}(10), 707--710 (2007)

\bibitem{chu1955bounds}
Chu, J.T.: On bounds for the normal integral.
\newblock Biometrika \textbf{42}(1/2), 263--265 (1955)

\bibitem{donoho2006compressed}
Donoho, D.L.: Compressed sensing.
\newblock IEEE Trans. Inf. Theory \textbf{52}(4), 1289--1306 (2006)

\bibitem{donoho2001uncertainty}
Donoho, D.L., Huo, X.: Uncertainty principles and ideal atomic decomposition.
\newblock IEEE Trans. Inf. Theory \textbf{47}(7), 2845--2862 (2001)

\bibitem{fan2001variable}
Fan, J., Li, R.: Variable selection via nonconcave penalized likelihood and its
  oracle properties.
\newblock J. Am. Stat. Assoc. \textbf{96}(456), 1348--1360 (2001)

\bibitem{DCT2012coherence}
Fannjiang, A., Liao, W.: Coherence pattern--guided compressive sensing with
  unresolved grids.
\newblock SIAM J. Imag. Sci. \textbf{5}(1), 179--202 (2012)

\bibitem{goodman2005introduction}
Goodman, J.W.: Introduction to Fourier optics.
\newblock Roberts and Company Publishers (2005)

\bibitem{LHY}
Lange, K., Hunter, D., Yang, I.: Optimization transfer using surrogate
  objective functions.
\newblock J. Comput. Graph. Statist. \textbf{9}(1), 1--20 (2000)

\bibitem{louYHX14}
Lou, Y., Yin, P., He, Q., Xin, J.: Computing sparse representation in a highly
  coherent dictionary based on difference of $ {L_1} $ and $ {L_2 }$.
\newblock J. Sci. Comput. \textbf{64}(1), 178--196 (2015)

\bibitem{louYX16}
Lou, Y., Yin, P., Xin, J.: Point source super-resolution via non-convex l1
  based methods.
\newblock J. Sci. Comput. \textbf{68}, 1082--1100 (2016)

\bibitem{lv2009unified}
Lv, J., Fan, Y., et~al.: A unified approach to model selection and sparse
  recovery using regularized least squares.
\newblock Annals of Stat. \textbf{37}(6A), 3498--3528 (2009)

\bibitem{mammone83}
Mammone, R.J.: Spectral extrapolation of constrained signals.
\newblock J. Opt. Soc. Am. \textbf{73}(11), 1476--1480 (1983)

\bibitem{natarajan95}
Natarajan, B.K.: Sparse approximate solutions to linear systems.
\newblock SIAM J. Comput. \textbf{24}(2), 227--234 (1995)

\bibitem{ochs2015iteratively}
Ochs, P., Dosovitskiy, A., Brox, T., Pock, T.: On iteratively reweighted
  algorithms for nonsmooth nonconvex optimization in computer vision.
\newblock SIAM J. Imaging Sci. \textbf{8}(1), 331--372 (2015)

\bibitem{papoulisC79}
Papoulis, A., Chamzas, C.: Improvement of range resolution by spectral
  extrapolation.
\newblock Ultra. Imag. \textbf{1}(2), 121--135 (1979)

\bibitem{parikh2014proximal}
Parikh, N., Boyd, S., et~al.: Proximal algorithms.
\newblock Foundations and Trends{\textregistered} in Optimization
  \textbf{1}(3), 127--239 (2014)

\bibitem{l1dl2}
Rahimi, Y., Wang, C., Dong, H., Lou, Y.: A scale invariant approach for sparse
  signal recovery.
\newblock SIAM J. Sci. Comput. \textbf{41}(6), A3649--A3672 (2019)

\bibitem{santosaS86}
Santosa, F., Symes, W.W.: Linear inversion of band-limited reflection
  seismograms.
\newblock SIAM J. Sci. Stat. Comp. \textbf{7}(4), 1307--1330 (1986)

\bibitem{shen2012likelihood}
Shen, X., Pan, W., Zhu, Y.: Likelihood-based selection and sharp parameter
  estimation.
\newblock J. Am. Stat. Assoc. \textbf{107}(497), 223--232 (2012)

\bibitem{tibshirani96lasso}
Tibshirani, R.: Regression shrinkage and selection via the lasso.
\newblock J. R. Stat. Soc. Series B \textbf{58}(1), 267--288 (1996)

\bibitem{tran2017unified}
Tran, H., Webster, C.: Unified sufficient conditions for uniform recovery of
  sparse signals via nonconvex minimizations.
\newblock arXiv preprint arXiv:1710.07348  (2017)

\bibitem{l1dl2accelerated}
Wang, C., M., Y., Rahimi, Y., Lou, Y.: Accelerated schemes for the l1/l2
  minimization.
\newblock arXiv preprint arXiv:1905.08946  (2019)

\bibitem{Xu2012}
Xu, Z., Chang, X., Xu, F., Zhang, H.: $l_{1/2}$ regularization: A thresholding
  representation theory and a fast solver.
\newblock IEEE Trans. Neural Netw. Learn. Syst. \textbf{23}, 1013--1027 (2012)

\bibitem{yinEX14}
Yin, P., Esser, E., Xin, J.: Ratio and difference of $l_1$ and $l_2$ norms and
  sparse representation with coherent dictionaries.
\newblock Comm. Inf. Syst. \textbf{14}(2), 87--109 (2014)

\bibitem{yinLHX14}
Yin, P., Lou, Y., He, Q., Xin, J.: Minimization of $\ell_{1-2}$ for compressed
  sensing.
\newblock SIAM J. Sci. Comput. \textbf{37}(1), A536--A563 (2015)

\bibitem{zhang2010nearly}
Zhang, C.: Nearly unbiased variable selection under minimax concave penalty.
\newblock Ann. Stat. pp. 894--942 (2010)

\bibitem{zhangX17}
Zhang, S., Xin, J.: Minimization of transformed ${L_1}$ penalty: Closed form
  representation and iterative thresholding algorithms.
\newblock Comm. Math. Sci. \textbf{15}, 511--537 (2017)

\bibitem{zhangX18}
Zhang, S., Xin, J.: Minimization of transformed ${L_1 }$ penalty: theory,
  difference of convex function algorithm, and robust application in compressed
  sensing.
\newblock Math. Program. \textbf{169}(1), 307--336 (2018)

\bibitem{zhang2009multi}
Zhang, T.: Multi-stage convex relaxation for learning with sparse
  regularization.
\newblock In: Adv. Neural Inf. Proces. Syst., pp. 1929--1936 (2009)

\end{thebibliography}
